\documentclass[review,onefignum,onetabnum]{siamonline190516}

\usepackage{color}
\usepackage{amssymb}

\newcommand{\be}{\begin{eqnarray}}
\newcommand{\ben}{\begin{eqnarray*}}
\newcommand{\en}{\end{eqnarray}}
\newcommand{\enn}{\end{eqnarray*}}

\definecolor{rot}{rgb}{1.000,0.000,0.000}
\definecolor{blue}{rgb}{0.000,0.000,1.000}

\usepackage{lipsum}
\usepackage{mathrsfs}
\usepackage{amsfonts}
\usepackage{graphicx}
\usepackage{epstopdf}
\usepackage{algorithmic}
\ifpdf
  \DeclareGraphicsExtensions{.eps,.pdf,.png,.jpg}
\else
  \DeclareGraphicsExtensions{.eps}
\fi

\usepackage{enumitem}
\setlist[enumerate]{leftmargin=.5in}
\setlist[itemize]{leftmargin=.5in}

\newsiamremark{remark}{Remark}
\newsiamremark{hypothesis}{Hypothesis}
\crefname{hypothesis}{Hypothesis}{Hypotheses}
\newsiamthm{claim}{Claim}

\headers{Uniqueness in determining rectangular periodic structures}{Jianli Xiang and Guanghui Hu}

\title{Uniqueness in determining binary grating profiles and refractive indices with a single incoming wave \thanks{Submitted to the editors DATE.
}
}

\author{Jianli Xiang\thanks{Three Gorges Math Research Center, College of Science, China Three Gorges University, Yichang 443002, People's Republic of China
  (\email{xiangjianli@ctgu.edu.cn}).}
\and Guanghui Hu \thanks{Corresponding author: School of Mathematical Sciences and LPMC, Nankai University, Tianjin 300071, People's Republic of China (\email{ghhu@nankai.edu.cn}).}
}

\usepackage{amsopn}

\ifpdf
\hypersetup{
  pdftitle={uniqueness in determining rectangular periodic structures},
  pdfauthor={Jianli Xiang}
}
\fi

\nolinenumbers
\begin{document}

\maketitle

\begin{abstract}
 We investigate inverse diffraction problems for penetrable gratings in a piecewise constant medium. In the TE polarization case, it is proved that a binary grating profile together with the refractive index beneath it can be uniquely determined by the near-field observation data incited by a single plane wave and measured on a line segment above the grating. Our approach relies on the expansion of solutions to the Helmholtz equation and the corner singularity analysis of solutions to the inhomogeneous Laplace equation with a piecewise continuous source term in a sector.  This paper also contributes to corner scattering theory for the Helmholtz equation in a special non-convex domain.
\end{abstract}



\section{Introduction}

The time-harmonic scattering of acoustic, electromagnetic and elastic waves by periodic surfaces plays a role in many areas of applied physics and engineering. Optical diffraction gratings date from the nineteenth century, which have a long history since Rayleigh's work \cite{Ray1907} published in 1907. We refer to the books \cite{BL, Petit1980, Wilcox1984} for its physical and mathematical background and to \cite{AbNe, BS, CF, DF92} for earlier studies on Maxwell's equations. In the TE or TM polarization case,
well-posedness of the scattering problem has been sufficiently studied for transmission problems of the Helmhotz equation under additional conditions imposed on the incident wavenumber, scattering interface and material parameters; see e.g., \cite{Tilo-hab, BS, D93, ES98, S98}.  The inverse scattering problem of recovering an unknown grating profile from the scattered field is of great practical importance, e.g., in quality control and design of diffractive elements with prescribed far-field patterns (\cite{Bao2001, D93, ES98, SK97, Turunen1997}). Since the uniqueness issue plays a significant role in such inverse problems, the purpose of this article is to present a complete answer to the problem of recovering a penetrable rectangular grating profile together with the material parameter from near-field observations of the scattered field. It is supposed that a binary grating remains invariant along one surface direction and we consider the TE polarization case. The media divided by the grating are supposed to be piecewise homogeneous and isotropic, and the measurement data are excited by a single plane wave only.

For perfectly reflecting periodic curves, there has been many uniqueness results in the literature.  In the TE polarization case (Dirichlet boundary condition), we refer to \cite{Bao1994, Kirsch1994} for the uniqueness results with one plane wave if the background medium is lossy and using  infinitely many quasi-periodic incident waves in non-absorbing media. Hettlich and Kirsch \cite{Hettlich1997} had proved that a finite number of incident plane waves with a fixed direction and distinct frequencies are sufficient to uniquely identify a $C^2$-smooth periodic curve, provided the grating height is a priori known. This has extended Schiffer's idea from inverse scattering by bounded obstacles to periodic structures. In the special case of piecewise linear surfaces, one can obtain global uniqueness results within the class of polygonal/polyhedral grating profiles by using a minimal number of incident planes. The first result in this respect was shown in \cite{Elschner2003} within rectangular periodic structures under the Dirichlet or Neumann boundary condition. In one of the author's work \cite{Elschner2010}, all periodic polygonal structures that cannot be identified by one incident plane wave were characterized and classified. Consequently, one can get a global uniqueness with at most four incident angles for recovering polygonal periodic structures in the Rayleigh frequency case. This was inspired by the reflection principle for the Helmholtz equation with the Dirichlet or Neumann boundary condition on a straight line and the dihedral theory for classifying unidentifiable bi-periodic structures in optics \cite{Bao2014}.

Kirsch's uniqueness result \cite{Kirsch1994} was extended to penetrable periodic layers in \cite{Strycharz1999}, where the author had proved that the grating profile together with the constitutive parameters can be completely determined from the scattered waves for all quasi-periodic incident waves. Elschner and Yamamoto \cite{Elschner2004} proved that multi-frequency near-field measurements can uniquely determine a penetrable grating profile in a piecewise constant medium. If the grating height is a priori known, a finite number of frequencies are sufficient to imply uniqueness. This can be considered as another extension of Schiffer's idea to periodic structures, in addition to the aforementioned work \cite{Hettlich1997}. Note that the measurements in \cite{Elschner2004,Strycharz1999} must be taken both above and below the periodic structure. Yang and Zhang \cite{Yang2012} showed that a smooth dielectric grating interface can be uniquely recovered by the scattered field measured only on  above the grating. Their proof is mainly based on the analogue of mixed reciprocity relation in periodic structures.

In this paper, we restrict our discussions to penetrable periodic surfaces of rectangular type in a piecewise constant medium in $\mathbb{R}^{2}$. Binary gratings have many applications in industry, because they can be easily fabricated \cite{SK97}.  There are two features of our uniqueness result. i) The measurement data are taken above the grating only and are excited by a single plane wave with an arbitrarily fixed direction and frequency. With one incoming wave, the inverse problem becomes more ill-posed and is thus more challenging. ii) Not only the binary grating profile but also the material parameter can be uniquely recovered, due to the delicate singularity analysis around a corner point. From numerical point of view,  our result ensures the existence of a unique global minimizer in the optimal design of penetrable binary gratings  with a constant refractive index (see e.g., \cite{D93, ES98}) from prescribed/measured near-field data.

It should be remarked that the uniqueness proof for perfectly reflecting surfaces (\cite{Bao2014, Elschner2010, Elschner2003}) cannot be applied to penetrable gratings, due to the lack of a corresponding reflection principle for treating the transmission conditions.   Our approach to the uniqueness is based on the expansion of analytic solutions to the Helmholtz equation and the corner singularity analysis of solutions to the inhomogeneous Laplace equation in weighted H\"older spaces. This is motivated by the recent scattering theory for bounded (non-periodic) inhomogeneous media with a singularity on the contrast support and for polygonal source terms (see e.g. \cite{ BPJ, Elschner2018, Hu2020, KS, PSV17}). However, the corner scattering theory applies only to convex domains so far. In this paper, we need to consider two distinct rectangular structures with the same corners, which bring essential difficulties as in justifying the corner scattering theory in a non-convex domain. Thanks to the rectangular nature of the scattering surface, we can adapt the singularity analysis performed in (\cite{Elschner2018}) to penetrable grating structures with right angles. Moreover, since the corner  singularity of the wave fields  relies heavily on material parameters, we prove that the constant refractive index beneath the grating can be uniquely identified once the grating profile has been recovered.

The rest of the paper is organized as follows. In section \ref{sec2}, mathematical formulations and main results are presented for grating diffraction problems in the TE polarization case. In section \ref{sec3}, we give some preliminaries and prepare several important lemmas for the uniqueness result. Sections \ref{Shape} and \ref{medium} are devoted to uniqueness proofs for shape identification and medium recovery, respectively. In the appendix, we present a proof to the well-posedness of forward scattering problem under more general transmission conditions. Finally, some concluding remarks will be made in section \ref{appendix}.

\section{Mathematical formulation and main result}\label{sec2}

Consider the TE-polarization of time-harmonic electromagnetic scattering of a plane wave from a penetrable binary grating which remains invariant along one surface direction $x_3$.  The media separated by the grating are supposed to be piecewise constant and non-absorbing. In two dimensions,  the cross-section $\Lambda$ of the grating surface in the $ox_1x_2$-plane is of rectangular type, i.e.,  neighboring line segments are always perpendicular to the $x_1$- and $x_2$- axies. More precisely, define a set $\mathcal{A}$ of all possible grating profiles by:
\begin{align*}
\mathcal{A}=\big\{\Lambda~|~&\Lambda\mbox{ is a non-self-intersecting  curve in $\mathbb{R}^{2}$ which is $2\pi$-periodic  in $x_1$. } \\ &\mbox{$\Lambda$ is piecewise linear and any linear part is parallel to the $x_{1}$- or $x_{2}$-axis} \big\},
\end{align*}
then we call a piecewise linear curve $\Lambda\in\mathcal{A}$ a rectangular profile (see the following Figure \ref{fig0}).

Denote by $\Omega_{\Lambda}^{+}$ ($\Omega_{\Lambda}^{-}$) the unbounded periodic domain over (below) $\Lambda$, that is the component of $\mathbb{R}^{2}$ separated by $\Lambda$ which is connected to $x_{2}=+\infty$ ($x_{2}=-\infty$). Let $\nu=(\nu_1, \nu_2)\in \mathbb{S}:=\{x\in \mathbb{R}^2: |x|=1\}$ be the normal direction at $\Lambda$ pointing into $\Omega^+_\Lambda$. We always suppose that $\nu_2\geq 0$, which is equivalent to the geometrical condition that
\begin{equation}\label{gc}
(x_1, x_2)\in \Omega^+_\Lambda\quad \Rightarrow\quad (x_1, x_2+s)\in \Omega^+_\Lambda\quad\mbox{for all}\quad s>0.
\end{equation}
The condition \eqref{gc} has been used in \cite{CM} for proving well-posedness of rough surface scattering problems with the Dirichlet boundary condition.
Suppose that a plane wave in the $(x_{1},x_{2})$-plane given by
\begin{equation*}
u^{i}(x_{1},x_{2})=e^{i\alpha x_{1}-i\beta x_{2}}, \quad \alpha=k_{1}\sin\theta, \quad \beta=k_{1}\cos\theta
\end{equation*}
with some incident angle $\theta\in(-\pi/2, \pi/2)$ and wave number $k_{1}>0$, is incident upon the grating $\Lambda$ from the top. Then the direct transmission scattering problem is to find the total field $u=u(x_{1},x_{2})$ such that
\begin{equation} \label{a}
\left\{\begin{array}{lll}
 \Delta u+k_{1}^{2}u=0,& \quad \mbox{in} \quad \Omega_{\Lambda}^{+}, \vspace{0.3cm} \\
 \Delta u+k_{2}^{2}u=0,& \quad \mbox{in} \quad \Omega_{\Lambda}^{-}, \vspace{0.3cm} \\
 \big[u\big]=\big[\frac{\partial u}{\partial\nu}\big]=0,  & \quad \mbox{on}\quad \Lambda,\vspace{0.3cm} \\
 u=u^{i}+u^{s},& \quad \mbox{in} \quad \Omega_{\Lambda}^{+},
 \end{array}\right.
\end{equation}
with the following radiation conditions as $x_{2}\rightarrow\pm\infty$:
\begin{equation} \label{rad1}
u^{s}=\sum_{n\in\mathbb{Z}}A_{n}^{+}~e^{i\alpha_{n}x_{1}+i\beta_{n}^{+}x_{2}},\quad \mbox{for}~x_{2}>\Lambda^{+}:=\max_{(x_{1},x_{2})\in\Lambda}x_{2},
\end{equation}
\begin{equation} \label{rad2}
u=\sum_{n\in\mathbb{Z}}A_{n}^{-}~e^{i\alpha_{n}x_{1}-i\beta_{n}^{-}x_{2}},\quad \mbox{for }~x_{2}<\Lambda^{-}:=\min_{(x_{1},x_{2})\in\Lambda}x_{2},
\end{equation}
where $\alpha_{n}:=n+\alpha$ and
\begin{equation*}
\beta_{n}^{+}:=\left\{\begin{array}{lll}
\sqrt{k_{1}^{2}-\alpha_{n}^{2}}\quad &\mbox{if}~~|\alpha_{n}|\leq k_{1},\vspace{0.3cm}\\
i\sqrt{\alpha_{n}^{2}-k_{1}^{2}} \quad &\mbox{if}~~|\alpha_{n}|> k_{1};
\end{array}\right. \quad\quad
\beta_{n}^{-}:=\left\{\begin{array}{lll}
\sqrt{k_{2}^{2}-\alpha_{n}^{2}}\quad &\mbox{if}~~|\alpha_{n}|\leq k_{2},\vspace{0.3cm}\\
i\sqrt{\alpha_{n}^{2}-k_{2}^{2}} \quad &\mbox{if}~~|\alpha_{n}|> k_{2}.
\end{array}\right.
\end{equation*}
In \eqref{a}, the notation $[\cdot]$ stands for the jumps of $u$ and $\partial_\nu u$ on the grating interface $\Lambda$. The expansions in (\ref{rad1}) and (\ref{rad2}) are the well-known Rayleigh expansions (see e.g. \cite{AbNe, DF92, Kirsch1993, Ray1907}); $A_{n}^{\pm}\in\mathbb{C}$ are called the Rayleigh coefficients. Throughout this paper we suppose that $k_2>0$ and $k_2\neq k_1$.
The series \eqref{rad1} and \eqref{rad2} together with their derivatives are uniform convergent in any compact set in $x_2>\Lambda^+$ and $x_2<\Lambda^-$, respectively, because
$u\in H_\alpha^1(S_H)$ (see below for the definition) and
the scattered and transmitted fields consist of infinitely many surface waves which exponentially decay as $x_2\rightarrow \pm \infty$.

\begin{figure}[h] \label{fig0}
 \begin{center}
  \includegraphics[width=7cm,height=6cm]{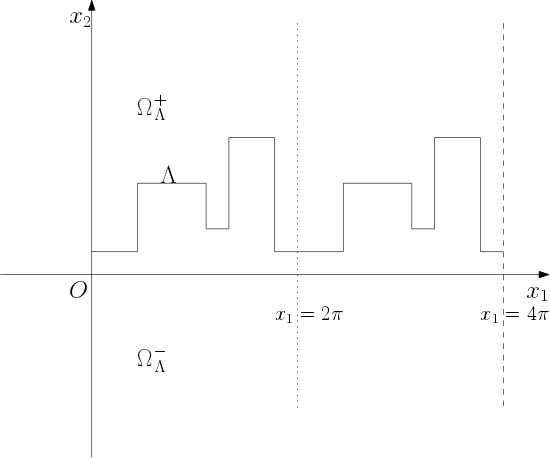}
  \end{center}
  \caption{Rectangular periodic structures.}
\end{figure}

Well-posedness of the above scattering problem \eqref{a}--\eqref{rad2} can be justified via standard variational arguments for weak solutions in the $\alpha$-quasiperiodic Sobolev space
\begin{equation*}
H^1_\alpha(S_H):=\big\{u\in H^1_{loc}(S_H),~e^{-i\alpha x_1}u\;\mbox{is $2\pi$-periodic in $x_1$}\big\},
\end{equation*}
with $S_H:=\{x\in \mathbb{R}^2: |x_2|<H\}$ for any $H>\max\{|\Lambda^+|, |\Lambda^-|\}$; see the appendix for the proof. In particular, uniqueness follows from Rellich's identifies with the factor $(x_2-c)\partial_2 \overline{u}$ for some $c\in \mathbb{R}$ applied to $S_H$, under the conditions that $k_2\neq k_1$ and the second component of the normal direction on $\Lambda$ is non-negative. In the literature (see \cite[Theorem 2.40]{Tilo-hab} and \cite{S98}), uniqueness was proved for interfaces given by a H\"older continuous graph, which can be weakened to the class of rectangular penetrable gratings considered in this paper.

Now we formulate the inverse problem with a single measurement data above the grating as follows. Let $b>\Lambda^{+}$ be a fixed constant and suppose $u=u(x_{1},x_{2})$ is a solution to the direct problem (\ref{a})--(\ref{rad2}). Determine the periodic interface $\Lambda\in\mathcal{A}$ from  knowledge of the near-field data $u(x_{1},b)$ for all $0<x_{1}<2\pi$.

The aim of this paper is to prove uniqueness in recovering a penetrable rectangular grating profile $\Lambda\in\mathcal{A}$ and the constant material parameter $k_2$ beneath $\Lambda$ with the arbitrarily fixed incident direction $\theta\in(-\pi/2, \pi/2)$ and wave number $k_{1}>0$. For brevity we denote by $(\Lambda, k_2)$ the shape and refractive index to be recovered. We are ready to state the main uniqueness result.
\begin{theorem} \label{Main}
Let $(\Lambda_{1}, k_{1,2}),(\Lambda_{2}, k_{2,2})$ be two penetrable rectangular gratings such that
\begin{itemize}
\item[(i)] $\Lambda_1, \Lambda_2\in \mathcal{A}$;
\item[(ii)] either $k_{1,2}>k_1>0$, $k_{2,2}>k_1>0$, or $0<k_{1,2}<k_1$, $0<k_{2,2}<k_1$.
\end{itemize}
Let $u_{1}$, $u_{2}$ be the unique solutions to the direct diffraction problem (\ref{a})--(\ref{rad2}) for $(\Lambda_{1}, k_{1,2})$, $(\Lambda_{2}, k_{2,2})$, respectively. If
\begin{equation}
u_{1}(x_{1},b)=u_{2}(x_{1},b)\quad \mbox{for all }x_{1}\in(0,2\pi),
\end{equation}
where $b>\max\{\Lambda_{1}^{+},\Lambda^{+}_{2}\}$ is a fixed constant, then $\Lambda_{1}=\Lambda_{2}$ and $k_{1,2}=k_{2,2}$.
\end{theorem}

\section{Preliminary lemmas}\label{sec3}

In this section, we will present some lemmas and corollaries to prepare for the proof of Theorem \ref{Main}, which are also interesting on their own right.

We begin with some notations to be used throughout the whole paper. Let $(r,\theta)$ with $\theta\in(-\pi,\pi]$, $r\geq0$ be the polar coordinates of $x=(x_{1},x_{2})$ in $\mathbb{R}^{2}$ and define
\ben
&&\Pi_{R}^{+}:=\{(r,\pi):0\leq r\leq R\}=\{(x_{1},x_{2}):x_{2}=0,~-R\leq x_{1}\leq 0\},\\
&&\Pi_{R}:=\{(r,\pi/2):0\leq r\leq R\}=\{(x_{1},x_{2}):x_{1}=0,~0\leq x_{2}\leq R\},\\
&&\Pi_{R}^{-}:=\{(r,0):0\leq r\leq R\}=\{(x_{1},x_{2}):x_{2}=0,~0\leq x_{1}\leq R\},\\
&&\Sigma_{R}^{+}:=\{(r,\theta):0<r<R,~\pi/2<\theta<\pi\}, \\
&&\Sigma_{R}^{-}:=\{(r,\theta):0<r<R,~0<\theta<\pi/2\}.
\enn
Obviously, $\Sigma_{R}^{+}\cup\Sigma_{R}^{-}\cup\Pi_{R}^{+}\cup\Pi_{R}\cup\Pi_{R}^{-}$ is a semicircle centered at origin with radius $R$. Let $B_{R}$ denote a disk centered at origin with radius $R$ and let  $\theta_{0}\in(0,\pi)$ be a fixed angle.  Define
\ben
&&B_{R,\theta_{0}}^{+}:=\big\{(r,\theta):-\theta_{0}<\theta<\theta_{0},~0<r<R\big\}, \quad B_{R,\theta_{0}}^{-}:=B_{R}\backslash\overline{B_{R,\theta_{0}}^{+}},\\
&&\Pi_{R,\theta_{0}}:=\big\{(r,\theta_0)\cup (r, -\theta_0):~0\leq r\leq R\big\}.
\enn

\begin{lemma} \label{Th1}
Let $\kappa_{1}$ and $\kappa_{2}$ be two (complex) constants in $B_{R}$. Assume that $v_{1}$ and $v_{2}$ satisfy the Helmholtz equations
\begin{equation*}
\Delta v_{1}+\kappa_{1}v_{1}=0,\quad \Delta v_{2}+\kappa_{2}v_{2}=0, \quad\mbox{in}~B_{R},
\end{equation*}
subject to the transmission conditions
\begin{equation*}
v_{1}=v_{2}, \quad \frac{\partial v_{1}}{\partial\nu}=\frac{\partial v_{2}}{\partial\nu}, \quad \mbox{on}~\Pi_{R}^{-}\cup\Pi_{R}.
\end{equation*}
If $\kappa_{1}\neq \kappa_{2}$, then $v_{1}=v_{2}\equiv0$ in $B_{R}$.
\end{lemma}
It should be noted that Lemma \ref{Th1} is a special case of Proposition 2.1 in \cite{ElschnerJ2015}, we omit the detailed proof in this paper. Slightly modifying Lemma \ref{Th1}, we can obtain the following result.

\begin{lemma} \label{Th2}
Suppose that $f_{1}\equiv0$ in $B_{R}\backslash\overline{\Sigma_{R}^{-}}$, $f_{1}$ is a constant different from zero in $\overline{\Sigma_{R}^{-}}$ and that $\kappa>0$ is a constant. Let $v_{1}$, $v_{2}\in H^2(B_R)$ be solutions to
\begin{align*}
\Delta v_{1}+\kappa^{2}(1+f_{1})v_{1}=0\quad \mbox{in}~B_{R}, \qquad
\Delta v_{2}+\kappa^{2}v_{2}=0\quad \mbox{in}~B_{R},
\end{align*}
subject to the transmission conditions
\begin{equation*}
v_{1}=v_{2}, \quad \frac{\partial v_{1}}{\partial\nu}=\frac{\partial v_{2}}{\partial\nu}, \quad \mbox{on}\quad \Pi_{R}^{-}\cup\Pi_{R}.
\end{equation*}
Then $v_{1}=v_{2}\equiv0$ in $B_{R}$.
\end{lemma}
\begin{proof}
Set $\kappa_{1}:=\kappa^{2}(1+f_{1})$ in $\overline{\Sigma_{R}^{-}}$. Then $\kappa_{1}$ is a constant different from $\kappa^{2}$ and $\Delta v_1+\kappa_1^2v_1=0$ in $\Sigma_R^-$. Since $v_2$ is analytic in $B_R$,  the Cauchy data of $v_{1}$ on $\Pi_R^-$ and $\Pi_R$ are analytic by the transmission boundary conditions. By the Cauchy-Kowalewski theorem and Holmgren's theorem, we can find a solution $\widetilde{v}_{1}$ to the following Cauchy problem in a piecewise analytic domain  (see e.g., \cite[Theorem 2.1]{LHY})
\begin{equation*}
\left\{\begin{array}{lll}
 \Delta\widetilde{v}_{1}+\kappa_{1}\widetilde{v}_{1}=0,& \quad \mbox{in} \quad B_{\varepsilon}\backslash\overline{\Sigma_{\varepsilon}^{-}}, \vspace{0.2cm} \\
 \widetilde{v}_{1}=v_{1},~\frac{\partial \widetilde{v}_{1}}{\partial \nu}=\frac{\partial v_{1}}{\partial \nu},  & \quad \mbox{on}\quad\Pi_{\varepsilon}^{-}\cup\Pi_{\varepsilon},
 \end{array}\right.
\end{equation*}
for some $0<\varepsilon<R$. Set $w_{1}:=\widetilde{v}_{1}$ in $B_{\varepsilon}\backslash\overline{\Sigma_{\varepsilon}^{-}}$,  $w_{1}:=v_{1}$ in $\Sigma_{\varepsilon}^{-}$
and $\kappa_{2}:=\kappa^{2}$.
It then follows that
\begin{equation*}
\left\{\begin{array}{lll}
 \Delta w_{1}+\kappa_{1}w_{1}=0,& \quad \mbox{in} \quad B_{\varepsilon},\vspace{0.2cm}\\
 \Delta v_{2}+\kappa_{2}v_{2}=0,& \quad  \mbox{in} \quad B_{\varepsilon}, \vspace{0.2cm}\\
 w_{1}=v_{2},\quad \frac{\partial w_{1}}{\partial \nu}=\frac{\partial v_{2}}{\partial \nu},  & \quad \mbox{on} \quad \Pi_{\varepsilon}^{-}\cup\Pi_{\varepsilon}.
 \end{array}\right.
\end{equation*}
Applying Lemma \ref{Th1}, we obtain $w_{1}=v_{2}\equiv0$ in $B_{\varepsilon}$. This together with the unique continuation leads to $v_{1}\equiv0$ in $B_{R}$. The proof is complete.
\end{proof}

Next, we investigate the asymptotic behavior of solutions to an inhomogeneous Laplacian equation in the disk $B_{R}$.
\begin{lemma} \label{A1}
Consider the inhomogeneous Laplace equation
\begin{equation*}
\left\{\begin{array}{lll}
  \Delta u=f,& \quad \mbox{in}\quad B_{R,\theta_{0}}^{\pm}, \\
  \big[u\big]=\big[\frac{\partial u}{\partial\nu}\big]=0, & \quad \mbox{on}\quad \Pi_{R,\theta_{0}},
\end{array}\right.
\end{equation*}
where $f\in C^{0,\delta}(B_{R,\theta_{0}}^{\pm})$ ($0<\delta<1$) and $f(r,\theta)\sim C^{\pm}r^{m}$ in $B_{R,\theta_{0}}^{\pm}$ as $r\rightarrow0^{+}$, with $m\geq0$ and $C^{\pm}\in \mathbb{C}$. Then
\begin{equation}\label{series}
u(r,\theta)=\sum_{n\geq0}r^{n}\big[a_{n}\sin(n\theta)+b_{n}\cos(n\theta)\big]+\mathcal{O}(r^{m+2}),\quad r\rightarrow0^{+},
\end{equation}
where $a_{n},b_{n}\in\mathbb{C}$ are such that the series in \eqref{series} is uniform convergent near the origin.
\end{lemma}
\begin{proof}
Write $u_{0}(r,\theta)=\sum\limits_{n\geq0}r^{n}\big[a_{n}\sin(n\theta)+b_{n}\cos(n\theta)\big]$. Then $u_0$ is a general solution to the homogeneous equation $\Delta u_{0}=0$ in $B_{R}$. Since $u\in H^2(B_R)$, we make the ansatz that
\begin{equation} \label{equ1}
u(r,\theta)-u_{0}(r,\theta)=\sum_{n\geq0}f_{n}(r)e^{in\theta},\qquad
f_n(r):=\frac{1}{2\pi}\int_{0}^{2\pi} (u-u_0)e^{-in\theta}\,d\theta.
\end{equation}
Inserting (\ref{equ1}) into the equation $\Delta u=f$, we find that
\begin{equation*}
f(r,\theta)=\Delta u_{0}(r,\theta)+\Delta\Big(\sum_{n\geq0}f_{n}(r)e^{in\theta}\Big) 
=\sum_{n\geq0}\Big[\frac{1}{r}(rf_{n}^{'})^{'}-\frac{n^{2}}{r^{2}}f_{n}\Big]e^{in\theta}.
\end{equation*}
Multiplying a term $e^{-in\theta}$ and integrating with respect to $\theta$ on both sides yield
\begin{equation*}
\frac{1}{r}(rf_{n}^{'})^{'}-\frac{n^{2}}{r^{2}}f_{n}=\widetilde{f}_{n} :=\frac{1}{2\pi}\int_{-\pi}^{\pi}f(r,\theta)e^{-in\theta}{\rm d}\theta.
\end{equation*}
Since
\begin{equation*}
2\pi\widetilde{f}_{n}=\int_{-\theta_{0}}^{\theta_{0}}f(r,\theta)e^{-in\theta}{\rm d}\theta +\Big(\int_{-\pi}^{-\theta_{0}}+\int_{\theta_{0}}^{\pi}\Big) f(r,\theta)e^{-in\theta}{\rm d}\theta,
\end{equation*}
we conclude from our assumption on $f$ that $\widetilde{f}_{n}(r,\theta)\sim Cr^{m}$ as $r\rightarrow0^{+}$. Hence, ${f}_{n}(r)\sim Cr^{m+2}$ as $r\rightarrow0^{+}$ for all $n\geq0$, which completes the proof.
\end{proof}
\vspace{0.2cm}

Based on the above Lemma \ref{A1}, we obtain the following corollary.
\begin{corollary} \label{Cor1}
Consider the transmission problem:
\begin{equation*}
\left\{\begin{array}{lll}
  \Delta u^{\pm}+k^{2}_{\pm}u^{\pm}=0,& \quad \mbox{in}\quad B_{R,\theta_{0}}^{\pm}, \\
  u^{+}=u^{-},\quad \frac{\partial u^{+}}{\partial\nu}=\frac{\partial u^{-}}{\partial\nu}, & \quad \mbox{on}\quad \Pi_{R,\theta_{0}},
\end{array}\right.
\end{equation*}
and define $u:=u^{+}$ in $B_{R,\theta_{0}}^{+}$, $u:=u^{-}$ in $B_{R,\theta_{0}}^{-}$. Then the function $u\in H^2(B_R)$ takes the asymptotic form
\begin{equation}\label{asy0}
u=\sum_{n\geq0}r^{n}\big[a_{n}\sin(n\theta)+b_{n}\cos(n\theta)\big]+\mathcal{O}(r^{2})\quad \mbox{as }r\rightarrow0^{+}, ~a_{n},b_{n}\in\mathbb{C}.
\end{equation}
Furthermore, if $u\not\equiv0$ in $B_{R}$,  we can write \eqref{asy0} as
\begin{equation}\label{asy}
u=\sum_{n\geq m}r^{n}\big[a_{n}\sin(n\theta)+b_{n}\cos(n\theta)\big]+\mathcal{O}(r^{m+2})\quad \mbox{as }r\rightarrow0^{+}, ~a_{n},b_{n}\in\mathbb{C},
\end{equation}
for some $m\geq0$ such that $|a_m|+|b_m|\neq 0$.
\end{corollary}
\begin{remark}
The relation \eqref{asy}  means that the lowest order expansion of $u$ is harmonic.
\end{remark}
\begin{proof}
We rewrite the equation for $u$ as $\Delta u=f$ in $B_{R}$, where $f:=-k^{2}_{+}u^{+}$ in $B_{R,\theta_{0}}^{+}$ and $f:=-k^{2}_{-}u^{-}$ in $B_{R,\theta_{0}}^{-}$. Since $f\in L^{2}(B_{R})$, we have $u\in H^{2}(B_{R})$, which is compactly imbedded into both $C^{0,\delta}(B_{R,\theta_{0}}^{+})$  and $C^{0,\delta}(B_{R,\theta_{0}}^{-})$ for some $0<\delta<1$. Applying Lemma \ref{A1}, we get the relation \eqref{asy0}. This also proves \eqref{asy} for $m=0$. If $u\sim C^\pm r^m$ as $r\rightarrow0$ in $B_{R, \theta_0}^\pm$ for some $m\geq 1$ and $C^\pm\in\mathbb{C}$, then  $f\sim -k^2_\pm C^\pm r^m$ near the origin and applying Lemma \ref{A1} again yields \eqref{asy}.
\end{proof}

To carry out the proof of Theorem \ref{Main}, we need to analyze the singularity of the inhomogeneous Laplacian equation in the semicircle $B_R\cap\{x_2>0\}$ with a piecewise continuous right term defined on $\Sigma_R^\pm$ and with the Dirichlet or Neumann boundary condition on $\Pi_R^\pm$. For this purpose, we construct a special solution to the Dirichlet problem (\ref{equ2}) or the Neumann problem (\ref{equ4}) when the right hand side is given by a homogeneous polynomial. Here and below, the notation $q_{k}$ denotes a homogeneous polynomial of order $k\geq0$ and the generic constants are denoted by $c$ or $c^{\pm}$ which may vary from line to line.  The proof of the following result is motivated by \cite[Lemma 3.6, Chapter 2.3.4]{NP}.
\begin{lemma} \label{A3}
Consider the Dirichlet problem:
\begin{equation} \label{equ2}
\left\{\begin{array}{lll}
  \Delta u=c^{\pm}q_{k},& \quad \mbox{in} \quad \Sigma_{R}^{\pm}, \\ \big[u\big]=\big[\frac{\partial u}{\partial\nu}\big]=0, & \quad \mbox{on} \quad \Pi_{R},\\
   u=0,& \quad \mbox{on} \quad \Pi_{R}^{+}\cup\Pi_{R}^{-},
\end{array}\right.
\end{equation}
and the Neumann problem:
\begin{equation} \label{equ4}
\left\{\begin{array}{lll}
  \Delta u=c^{\pm}q_{k},& \quad \mbox{in} \quad \Sigma_{R}^{\pm}, \\ \big[u\big]=\big[\frac{\partial u}{\partial\nu}\big]=0, & \quad \mbox{on} \quad \Pi_{R},\vspace{0.05cm}\\
   \frac{\partial u}{\partial\nu}=0,& \quad \mbox{on} \quad \Pi_{R}^{+}\cup\Pi_{R}^{-}.
\end{array}\right.
\end{equation}
There exist a special solution to (\ref{equ2}) of the form
\begin{equation} \label{equ3}
u(r,\theta)=q_{k+2}^{\pm}(r,\theta)+C_{k, D}\; r^{k+2}\big\{\ln r\sin[(k+2)\theta] +\theta\cos[(k+2)\theta]\big\} \quad \mbox{in}~\Sigma_{R}^{\pm}
\end{equation}
for some $C_{k, D}\in \mathbb{C}$. In the Neumann case, a special solution to (\ref{equ4}) takes the form
\begin{equation}
u(r,\theta)=q_{k+2}^{\pm}(r,\theta)+C_{k, N}\;r^{k+2}\big\{\ln r\cos[(k+2)\theta]-\theta\sin[(k+2)\theta]\big\}  \quad \mbox{in }\Sigma_{R}^{\pm}
\end{equation}
for some $C_{k, N}\in \mathbb{C}$. Moreover, we have $C_{k, D}=C_{k, N}=0$ if $c^+=c^-=0$, and $q_{k+2}^\pm$ solve the same Dirichlet or Neumann problem in $\Sigma_R^\pm$.
\end{lemma}
\begin{proof} We only consider the Dirichlet boundary value problem. The Neumann case can be treated analogously. Write $c=C_{k, D}$, $q_{k}(r,\theta)=r^{k}p_{k}(\theta)$ and $q_{k+2}^{\pm}(r,\theta) =r^{k+2}f_{k}^{\pm}(\theta)$. To make $u(r,\theta)$ of the form (\ref{equ3}) a solution to (\ref{equ2}), we only need to require
\begin{equation} \label{equ}
\left\{\begin{array}{lll}
  [\partial_{\theta}^{2}+(k+2)^{2}]f_{k}^{\pm}(\theta)=c^{\pm}p_{k}(\theta), \quad \mbox{in}\quad\Sigma_{R}^{\pm},\vspace{0.1cm} \\
  f_{k}^{+}(\frac{\pi}{2})=f_{k}^{-}(\frac{\pi}{2}), \quad \partial_{\theta}f_{k}^{+}(\frac{\pi}{2}) =\partial_{\theta}f_{k}^{-}(\frac{\pi}{2}),\vspace{0.1cm} \\
  f_{k}^{-}(0)=0, \quad f_{k}^{+}(\pi)=(-1)^{k+1}c\pi,
\end{array}\right.
\end{equation}
because $r^{k+2}\big\{\ln r\sin[(k+2)\theta] +\theta\cos[(k+2)\theta]\big\}$ is a harmonic function for any $r>0$. The general solution $f_{k}^{\pm}(\theta)$ to the above differential equation can be written as
\begin{equation*}
f_{k}^{\pm}(\theta)=a^{\pm}\cos[(k+2)\theta]+b^{\pm}\sin[(k+2)\theta]+h_{k}^{\pm}(\theta),
\end{equation*}
where $h_{k}^{\pm}(\theta)$ are special solutions to
\begin{equation*}
(h_{k}^{\pm}(\theta))^{\prime\prime}+(k+2)^{2}h_{k}^{\pm}(\theta) =c^{\pm}p_{k}(\theta),\quad \theta\in(0,\pi/2)\cup(\pi/2,\pi).
\end{equation*}
Through simple calculations, we may suppose that
\begin{equation*}
h_{k}^{\pm}(\theta)=\frac{c^{\pm}}{k+2}\int_{0}^{\theta}\sin[(k+2)(\theta-\tau)] p_{k}(\tau){\rm d}\tau,\quad \theta\in(0,\pi/2)\cup(\pi/2,\pi).
\end{equation*}
To determine the coefficients $a^{\pm}$ and $b^{\pm}$, we use the transmission and the boundary conditions in (\ref{equ}) to get
\be   \label{p1}
&&a^{+}\cos\frac{(k+2)\pi}{2}-a^{-}\cos\frac{(k+2)\pi}{2}+b^{+}\sin\frac{(k+2)\pi}{2}
-b^{-}\sin\frac{(k+2)\pi}{2}=p_{1},
\\ \label{p2}
&&-a^{+}\sin\frac{(k+2)\pi}{2}+a^{-}\sin\frac{(k+2)\pi}{2}+b^{+}\cos\frac{(k+2)\pi}{2}
-b^{-}\cos\frac{(k+2)\pi}{2}=p_{2},
\\  \nonumber
&&a^{-}=0 \quad \mbox{and} \quad (-1)^{k}a^{+}=(-1)^{k+1}c\pi-h_{k}^{+}(\pi),
\en
where
\begin{equation*}
p_{1}:=(h_{k}^{-}-h_{k}^{+})\big|_{\frac{\pi}{2}},\quad p_{2}:=\frac{(h_{k}^{-}-h_{k}^{+})^{'}\big|_{\frac{\pi}{2}}}{k+2}.
\end{equation*}
Since $a^{-}=0$, by equations (\ref{p1}) and (\ref{p2}) we obtain that
\begin{equation*}
a^{+}=p_{1}\cos\frac{(k+2)\pi}{2}-p_{2}\sin\frac{(k+2)\pi}{2},\quad\quad
b^{+}-b^{-}=p_{1}\sin\frac{(k+2)\pi}{2}+p_{2}\cos\frac{(k+2)\pi}{2}.
\end{equation*}
Then we can choose a proper constant $c$ such that $-c\pi-h_{k}^{+}(\pi)=p_{1}\cos\frac{(k+2)\pi}{2}-p_{2}\sin\frac{(k+2)\pi}{2}$. Hence, the coefficients $a^{\pm}$ are uniquely determined and there exist infinite solutions $(b^{+},b^{-})$ satisfying the system (\ref{p1})--(\ref{p2}). On the other hand, it is obvious that $c=0$ if $c^\pm=0$. The proof is complete.
\end{proof}

\begin{lemma} \label{lem=}
Let $H_{n}(r, \theta)$ be a  harmonic polynomial of order $n$ in two dimensions. If the homogeneous polynomials $q^{\pm}_{n+2}$ ($n\geq0$) satisfy
\begin{equation}\label{3.7}
\left\{\begin{array}{lll}
  \Delta q^{\pm}_{n+2}=H_{n},&\quad\mbox{in}\quad \Sigma_{R}^{\pm},\\
  q^{+}_{n+2}=q^{-}_{n+2},\quad \frac{\partial q^{+}_{n+2}}{\partial \nu}=\frac{\partial q^{-}_{n+2}}{\partial \nu}, & \quad \mbox{on}\quad \Pi_{R}, \vspace{0.2cm}\\
  q^{\pm}_{n+2}=\frac{\partial q^{\pm}_{n+2}}{\partial \nu}=0, & \quad \mbox{on}\quad \Pi_{R}^{\pm}.
\end{array}\right.
\end{equation}
Then $q^{+}_{n+2}=q^{-}_{n+2}$.
\end{lemma}
\begin{proof}
Since $q_{n+2}^{\pm}$ is a homogeneous polynomial of order $n+2$, we can expand it into a convergent series in Cartesian coordinates:
\begin{equation*}
q_{n+2}^{\pm}=\sum\limits_{j=0}^{n+2}a_{j}^{\pm}x_{1}^{n+2-j}x_{2}^{j},\qquad n\geq0.
\end{equation*}
Below we shall prove that $a_j^+=a_j^-$ by using the transmission and boundary conditions in \eqref{3.7} together with the fact that $\Delta^2 q_{n+2}^\pm=\Delta H_n=0$.

In view of the transmission and boundary conditions,
\ben
&&q_{n+2}^{\pm}|_{x_{2}=0}=\frac{\partial q_{n+2}^{\pm}}{\partial x_{2}}\Big|_{x_{2}=0}=0,\\
&&q_{n+2}^{+}|_{x_{1}=0}=q_{n+2}^{-}|_{x_{1}=0}, \quad
\frac{\partial q_{n+2}^{+}}{\partial x_{1}}\Big|_{x_{1}=0}=\frac{\partial q_{n+2}^{-}}{\partial x_{1}}\Big|_{x_{1}=0},
\enn
we get $a_{0}^{\pm}=a_{1}^{\pm}=0$ and $a_{n+2}^{+}=a_{n+2}^{-}:=\widetilde{a}_{n+2}$, $a_{n+1}^{+}=a_{n+1}^{-}:=\widetilde{a}_{n+1}$. Hence,
\begin{equation*}
q_{n+2}^{\pm}=\sum_{j=2}^{n+2}a_{j}^{\pm}x_{1}^{n+2-j}x_{2}^{j}.
\end{equation*}

For $n=0$, we have $q_{2}^{+}=\widetilde{a}_{2}x_{2}^{2}=q_{2}^{-}$.

For $n=1$, we have $q_{3}^{+}=\widetilde{a}_{2}x_{1}x_{2}^{2}+\widetilde{a}_{3}x_{2}^{3}=q_{3}^{-}$.

For $n\geq2$, it is easy to see that
\begin{align*}
\Delta q_{n+2}^{\pm}=&\frac{\partial}{\partial x_{1}}\Big(\sum_{j=2}^{n+1}(n-j+2)a_{j}^{\pm}x_{1}^{n+1-j}x_{2}^{j}\Big) +\frac{\partial}{\partial x_{2}}\Big(\sum_{j=2}^{n+2}ja_{j}^{\pm}x_{1}^{n+2-j}x_{2}^{j-1}\Big) \\
=&\sum_{j=0}^{n}\big[(n-j+2)(n-j+1)a_{j}^{\pm}+(j+1)(j+2)a_{j+2}^{\pm}\big]x_{1}^{n-j}x_{2}^{j} \\
=&\sum_{j=0}^{n}b_{j}^{\pm}x_{1}^{n-j}x_{2}^{j},
 \end{align*}
where
\ben
b_{j}^{\pm}:=(n-j+2)(n-j+1)a_{j}^{\pm}+(j+1)(j+2)a_{j+2}^{\pm}.
\enn
Analogously,
\begin{equation*}
\Delta^{2}q_{n+2}^{\pm}=\sum_{j=0}^{n-2}d_{j}^{\pm}x_{1}^{n-2-j}x_{2}^{j}, \quad d_{j}^{\pm}:=(n-j)(n-j-1)b_{j}^{\pm}+(j+1)(j+2)b_{j+2}^{\pm}.
\end{equation*}
Since $\Delta^{2}q_{n+2}^{\pm}=0$, we have $d_{j}^{\pm}=0$ for $0\leq j\leq n-2$, which implies that
\begin{equation*}
\frac{(n-j-1)(n-j)a_{j+2}^{\pm}+(j+3)(j+4)a_{j+4}^{\pm}}{(n-j+1)(n-j+2)a_{j}^{\pm}+(j+1)(j+2)a_{j+2}^{\pm}} =\frac{b_{j+2}^{\pm}}{b_{j}^{\pm}}=-\frac{(n-j-1)(n-j)}{(j+1)(j+2)}.
\end{equation*}
Equivalently, we may rewrite the previously relation as
\ben
0&=&(j+4)(j+3)(j+2)(j+1) \;a_{j+4}+2 (j+2)(j+1)(n-j-1)(n-j)\; a_{j+2} \\
&&+ \,(n-j-1)(n-j)(n-j+1)(n-j+2)\; a_j,
\enn
where $a_j:=a_j^+-a_j^-$ for $0\leq j\leq n+2$. Since $a_0=a_1=0$ and $a_{n+1}=a_{n+2}=0$, the homogeneous linear system for $a_j$ ($2\leq j\leq n$) corresponds to the $(n-1)\times (n-1)$ matrix $D_{n-1}$:
\begin{equation*}
D_{n-1}=\left(\begin{array}{ccccccc}
\mathbb{B}_{0}(n) & 0 & \mathbb{C}_{0}(n) & 0 & \cdots & 0 & 0 \\
0 & \mathbb{B}_{1}(n) & 0 & \mathbb{C}_{1}(n) & \cdots & 0 & 0 \\
\mathbb{A}_{2}(n) & 0 & \mathbb{B}_{2}(n) & 0 & \cdots & 0 & 0 \\
0 & \mathbb{A}_{3}(n) & 0 & \mathbb{B}_{3}(n) & \cdots & 0 & 0 \\
\vdots & \vdots & \vdots & & \vdots \ddots & \vdots & \vdots \\
0 & 0 & 0 & 0 & \cdots & \mathbb{B}_{n-3}(n) & 0 \\
0 & 0 & 0 & 0 & \cdots & 0 & \mathbb{B}_{n-2}(n)
\end{array}\right)
\end{equation*}
where $\mathbb{A}_{j}(n):=(n-j-1)(n-j)(n-j+1)(n-j+2)$, $\mathbb{B}_{j}(n):=2 (j+2)(j+1)(n-j-1)(n-j)$ and $\mathbb{C}_{j}(n):=(j+4)(j+3)(j+2)(j+1)$.

For $n=2$, we have $\mathbb{B}_{0}(2)=8\neq0$; for $n=3$, we have $|D_{2}|=\mathbb{B}_{0}(3)\mathbb{B}_{1}(3)=24^{2}\neq0$; for $n\geq4$, we have
\begin{equation*}
|D_{n-1}|=\mathbb{B}_{0}(n)\mathbb{B}_{1}(n)\left(\mathbb{B}_{2}(n) -\frac{\mathbb{A}_{2}(n)}{\mathbb{B}_{0}(n)}\mathbb{C}_{0}(n)\right)\cdots\left(\mathbb{B}_{n-2}(n) -\frac{\mathbb{A}_{n-2}(n)}{\mathbb{B}_{n-4}(n)}\mathbb{C}_{n-4}(n)\right).
\end{equation*}
Note that $\mathbb{B}_{j}(n)\neq0$ ($0\leq j\leq n-2$, $n\geq4$). Since
\begin{align*}
&\mathbb{B}_{j}(n)\mathbb{B}_{j-2}(n)-\mathbb{A}_{j}(n)\mathbb{C}_{j-2}(n) \\
=&4(j+1)(j+2)(n-j-1)(n-j)(j-1)j(n-j+1)(n-j+2) \\
&-(n-j-1)(n-j)(n-j+1)(n-j+2)(j+2)(j+1)j(j-1) \\
=&3(j-1)j(j+1)(j+2)(n-j-1)(n-j)(n-j+1)(n-j+2)\\
\neq&0
\end{align*}
for any $2\leq j\leq n-2$,
we obtain that $|D_{n-1}|\neq0$. Consequently, there exists only one trivial solution to the homogeneous linear system for $a_j$ ($2\leq j\leq n$), that is $a_j=0$ ($2\leq j\leq n$).
Recalling the definition of $q_{n+2}^{\pm}$, we conclude that $q_{n+2}^{+}=q_{n+2}^{-}$. The proof is complete.
\end{proof}

Relying on the above preparations, we will prove the uniqueness result in Theorem \ref{Main}.  Firstly we prove $\Lambda_1=\Lambda_2$ in Section \ref{Shape} below, and then prove $k_{1,2}=k_{2,2}$ in Section \ref{medium}.

\section{Proof of Theorem \ref{Main}: determination of grating profiles} \label{Shape}

Since
\begin{equation*}
u_{1}(x_{1},b)=u_{2}(x_{1},b)\quad \mbox{ for all }x_{1}\in(0,2\pi),
\end{equation*}
we obtain that $u_{1}(x_{1},x_{2})=u_{2}(x_{1},x_{2})$ in $x_{2}>b$, and the unique continuation of solutions to the Helmholtz equation leads to
\begin{equation*}
u_{1}(x_{1},x_{2})=u_{2}(x_{1},x_{2})\quad \mbox{ for all }x\in \Omega_{\Lambda_{1}}^{+}\cap\Omega_{\Lambda_{2}}^{+}.
\end{equation*}
Assume on the contrary that $\Lambda_1\neq \Lambda_2$. Switching the notations for $\Lambda_1$ and $\Lambda_2$ if necessary, we consider the following  cases:
\begin{itemize}
  \item Case one: there exists a corner point $O$ of $\Lambda_{1}$ such that $O\in\Omega_{\Lambda_{2}}^{+}$ (see Figure \ref{fig1});

  \item Case two: all corners of $\Lambda_1$ and $\Lambda_2$ coincide but $\Lambda_1\neq \Lambda_2$ (see Figure \ref{fig2});

  \item Case three:  there exists a corner point $O$ of $\Lambda_{2}$ lying on $\Lambda_1$, but $O$ is not a corner of $\Lambda_1$ (see Figure \ref{fig3}).
\end{itemize}
Obviously, the first and last cases imply that the corners of $\Lambda_1$ and $\Lambda_2$ do not coincide completely. Using coordinate translation and rotation, we always suppose that the corner $O$ is located at the origin.

\subsection{Case one}\label{case-one}

\begin{figure}[h] \label{fig1}
 \begin{center}
  \includegraphics[width=7cm,height=3cm]{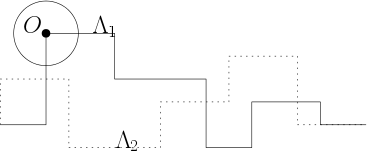}
  \end{center}
  \caption{Case one:  there exists a corner point $O$ of $\Lambda_{1}$ such that $O\in\Omega_{\Lambda_{2}}^{+}$.}
\end{figure}
Let $B_{R}$ denote a disk centered at the point $O$ with radius $R$ such that $B_{R} \subseteq \Omega_{\Lambda_{2}}^{+}$. Since this corner stays away from $\Lambda_{2}$ and belongs to $\Omega_{\Lambda_{2}}^{+}$, the function $u_{2}$ satisfies the Helmholtz equation with the wave number $k_{1}$ in $B_{R}$, while $u_{1}$ fulfills the Helmholtz equation with the variable potential $k_{1}^{2}(1+f_{1})$. Here, $f_{1}$ is a piecewise constant function defined by
\begin{equation*}
f_{1}:=\left\{\begin{array}{lll}
 0,& \quad \mbox{in}\quad B_{R}\cap\Omega_{\Lambda_{1}}^{+},  \vspace{0.2cm} \\
 (\frac{k_{1,2}}{k_{1}})^{2}-1,& \quad \mbox{in}\quad B_{R}\cap\Omega_{\Lambda_{1}}^{-}.
\end{array}\right.
\end{equation*}
Recalling the transmission conditions in (\ref{a}), we find that the pair $(u_{1},u_{2})$ is a solution to the following system:
\begin{equation*}
\left\{\begin{array}{lll}
\Delta u_{1}+k_{1}^{2}(1+f_{1})u_{1}=0,\quad &\mbox{in}\quad B_{R},\vspace{0.2cm}\\
\Delta u_{2}+k_{1}^{2}u_{2}=0,\quad &\mbox{in}\quad B_{R}, \vspace{0.2cm}\\
u_{1}=u_{2}, \quad &\mbox{on}\quad B_{R}\cap\Lambda_{1}, \vspace{0.2cm} \\
\frac{\partial u_{1}}{\partial\nu}=\frac{\partial u_{2}}{\partial\nu}, \quad &\mbox{on}\quad  B_{R}\cap\Lambda_{1}.
\end{array}\right.
\end{equation*}
Using Lemma \ref{Th2}, we obtain that $u_{1}=u_{2}\equiv0$ in $B_{R}$ and thus $u_{1}\equiv0$ in $\mathbb{R}^{2}$. To derive a contradiction we recall the Rayleigh expansion for $u_1$:
\begin{equation*}
u_{1}(x)=e^{i(\alpha x_{1}-\beta x_{2})}+\sum\limits_{n\in\mathbb{Z}}A_{n}^{+}e^{i(\alpha_{n}x_{1}+\beta_{n}^{+}x_{2})}, \qquad x_{2}\geq b
\end{equation*}
for some $b>\Lambda_1^+$. Taking $x_{2}=b$, we deduce from $u_1\equiv 0$ and $\alpha_{0}=\alpha$ that
\begin{equation*}
e^{i\alpha x_{1}}(e^{-i\beta b}+A_{0}^{+}e^{i\beta_{0}^{+}b}) +\sum\limits_{n\neq 0}A_{n}^{+}e^{i(n+\alpha)x_{1}}e^{i\beta_{n}^{+}b}=0,\quad \mbox{for all}\quad x_1\in \mathbb{R}.
\end{equation*}
Multiplying a term $e^{-i\alpha x_{1}}$ on both sides and integrating over $(0, 2\pi)$ with respect to $x_1$, we conclude that
\begin{equation*}
e^{-i\beta b}+A_{0}e^{i\beta b}=0,\quad A_{n}^{+}e^{i\beta_{n}^{+}b}=0,\quad\mbox{for all}\quad n\neq 0,
\end{equation*}
which yields $A_{0}=-e^{-2i\beta b}$ and $A_n=0$ if $n\neq 0$. Since $A_{0}\in\mathbb{C}$ is a constant, this is impossible for any $b>\Lambda_1^+$. This contradiction implies that $\Lambda_{1}=\Lambda_{2}$.

\subsection{Case two}
The corners of $\Lambda_{1}$ and $\Lambda_{2}$ coincide (see Figure \ref{fig2}), implying that $\Lambda_1$ and $\Lambda_2$ have the same height and also the same grooves but with different opening directions.
\begin{figure}[h] \label{fig2}
 \begin{center}
  \includegraphics[width=5.5cm,height=3cm]{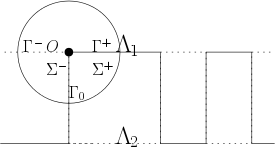}
  \end{center}
  \caption{Case two: corners of $\Lambda_1$ and $\Lambda_2$ are identical but $\Lambda_1\neq \Lambda_2$.}
\end{figure}

Choose a corner point $O\in\Lambda_{1}\cap\Lambda_{2}$ and $R>0$ sufficiently small such that the disk $B_{R}:=\{x\in \mathbb{R}^{2}:|x|<R\}$ does not contain other corners. Introduce the notations (see Figure \ref{fig2})
\begin{equation*}
B_{R}\cap \Lambda_{1}=\Gamma^{+}\cup\Gamma_{0},\quad B_{R}\cap \Lambda_{2}=\Gamma^{-}\cup\Gamma_{0},\quad \Sigma^{+}=B_{R}\cap \Omega_{\Lambda_{1}}^{-},\quad \Sigma^{-}=B_{R}\cap \Omega_{\Lambda_{2}}^{-}.
\end{equation*}
We can conclude that $u_{1},u_{2}\in H^{2}(B_{R})\cap C^{0,\delta}(B_{R})$ ($0<\delta<1$) fulfill the system
\begin{equation*}
\left\{\begin{array}{lll}
\Delta u_{1}+k_{1,2}^{2}u_{1}=0,\quad \Delta u_{2}+k_{1}^{2}u_{2}=0,\quad &\mbox{in}\quad \Sigma^{+}, \vspace{0.2cm} \\
\Delta u_{1}+k_{1}^{2}u_{1}=0,\quad  \Delta u_{2}+k_{2,2}^{2}u_{2}=0,\quad &\mbox{in}\quad \Sigma^{-}, \vspace{0.2cm}\\
\big[u_{1}\big]=\big[\frac{\partial u_{1}}{\partial\nu}\big]=0,\quad \big[u_{2}\big]=\big[\frac{\partial u_{2}}{\partial\nu}\big]=0, & \mbox{on}\quad \Gamma_{0}, \vspace{0.2cm}\\
u_{1}=u_{2}, \quad \frac{\partial u_{1}}{\partial\nu}=\frac{\partial u_{2}}{\partial\nu}, \quad & \mbox{on}\quad \Gamma^{+}\cup\Gamma^{-}.
\end{array}\right.
\end{equation*}
By Corollary \ref{Cor1}, we have
\begin{equation} \label{uj}
u_{j}(r,\theta)=\sum_{n\geq0}r^{n}\big[a_{n}^{(j)}\sin(n\theta)+b_{n}^{(j)}\cos(n\theta)\big]+\mathcal{O}(r^{2}),
\quad r\rightarrow0^{+}, ~j=1,2.
\end{equation}
Let $w=u_{1}-u_{2}$ in $\Sigma:=\Sigma^{+}\cup\Sigma^{-}\cup\Gamma_{0}$. Then we get a Cauchy problem for the Laplacian equation with an inhomogeneous source term:
\begin{equation}\label{inho}
\left\{\begin{array}{lll}
\Delta w=k_{1}^{2}u_{2}-k_{1,2}^{2}u_{1}:=f^{+},\quad &\mbox{in}\quad\Sigma^{+}, \vspace{0.2cm}\\
\Delta w=k_{2,2}^{2}u_{2}-k_{1}^{2}u_{1}:=f^{-},\quad &\mbox{in}\quad\Sigma^{-}, \vspace{0.2cm}\\
\big[w\big]=\big[\frac{\partial w}{\partial\nu}\big]=0,\quad &\mbox{on}\quad\Gamma_{0}, \vspace{0.2cm}\\
w=\frac{\partial w}{\partial\nu}=0, \quad &\mbox{on}\quad\Gamma^{+}\cup\Gamma^{-}.
\end{array}\right.
\end{equation}
Below we shall prove that the previous Cauchy problem is an overdetermined boundary value with the trivial solution only. We remark that the case $f^+=f^-$  has been considered in \cite{Elschner2018} where corner scattering theory in a convex domain has been discussed. Inspired by \cite{Elschner2018}, we need to study a special corner scattering problem with two right angles in this paper.
Our approach relies on the singularity analysis of the inhomogeneous Laplace equation with a piecewisely continuous right hand side in a semi-disk. We refer to the fundamental paper \cite{K1967} and the monographs \cite{KMR, MNP, NP} for a general regularity theory of elliptic boundary value problems in domains with non-smooth boundaries.

For clarity, we shall divide our proof in Case two into four steps.

{\bf Step 1:} Prove that $f^\pm(O)=0$ and $b_0^{(1)}=b_0^{(2)}=0$.

Since $f^\pm$ are H\"older continuous near $O$, we set $c_{0}^{\pm}:=f^{\pm}(O)$. Consider the Dirichlet and Neumann problems separately:
\begin{equation} \label{Dir}
\left\{\begin{array}{lll}
  \Delta v_{0,D}=c_{0}^{\pm},& \quad \mbox{in}~\Sigma^{\pm},\vspace{0.05cm}  \\
  \big[v_{0,D}\big]=\big[\frac{\partial v_{0,D}}{\partial \nu}\big]=0, & \quad \mbox{on} ~\Gamma_{0},\vspace{0.05cm} \\
  v_{0,D}=0, & \quad \mbox{on}~\Gamma^{+}\cup\Gamma^{-},
\end{array}\right.
\left\{\begin{array}{lll}
  \Delta v_{0,N}=c_{0}^{\pm},& \quad \mbox{in}~\Sigma^{\pm}, \vspace{0.05cm} \\
  \big[v_{0,N}\big]=\big[\frac{\partial v_{0,N}}{\partial \nu}\big]=0, & \quad \mbox{on}~\Gamma_{0}, \vspace{0.05cm}\\
  \frac{\partial v_{0,N}}{\partial\nu}=0, & \quad \mbox{on}~\Gamma^{+}\cup\Gamma^{-}.
\end{array}\right.
\end{equation}
where the right hand sides are given by the lowest order term of $f^\pm$.  By Lemma \ref{A3}, we know that there exist two special solutions to (\ref{Dir}) of the form
\ben
v_{0,D}&=&q^{\pm}_{2,D}(r,\theta)+C_{0,D}r^{2}\big[\ln r\sin(2\theta) +\theta\cos(2\theta)\big],\\
v_{0,N}&=&q^{\pm}_{2,N}(r,\theta)+C_{0,N}r^{2}\big[\ln r\cos(2\theta) -\theta\sin(2\theta)\big],
\enn
where $q^{\pm}_{2,D}(r,\theta)$ and $q^{\pm}_{2,N}(r,\theta)$ are homogeneous polynomials of degree two satisfying the system
\begin{equation*}
\left\{\begin{array}{lll}
  \Delta q^{\pm}_{2,D}=c_{0}^{\pm},& \quad \mbox{in}\quad \Sigma^{\pm}, \\
  q^{+}_{2,D}=q^{-}_{2,D},\quad & \quad \mbox{on}\quad \Gamma_{0}, \\
  \frac{\partial q^{+}_{2,D}}{\partial \nu}=\frac{\partial q^{-}_{2,D}}{\partial \nu}, &\quad\mbox{on}\quad\Gamma_{0},\vspace{0.05cm}\\
  q^{\pm}_{2,D}=0, & \quad \mbox{on}\quad \Gamma^{\pm}.
\end{array}\right. \quad \quad
\left\{\begin{array}{lll}
  \Delta q^{\pm}_{2,N}=c_{0}^{\pm},& \quad \mbox{in}\quad \Sigma^{\pm}, \\
  q^{+}_{2,N}=q^{-}_{2,N},& \quad \mbox{on}\quad \Gamma_{0}, \\
  \frac{\partial q^{+}_{2,N}}{\partial \nu}=\frac{\partial q^{-}_{2,N}}{\partial \nu},&\quad\mbox{on}\quad\Gamma_{0},\vspace{0.05cm} \\
  \frac{\partial q^{\pm}_{2,N}}{\partial \nu}=0, & \quad \mbox{on}\quad \Gamma^{\pm}.
\end{array}\right.
\end{equation*}

For $0<\delta<1$, $l\in\mathbb{N}$ and $\eta\in\mathbb{N}$, the weighted H\"older spaces $\Lambda^{l,\delta}_{\eta}(\Sigma)$ will be used to characterize the singularity of solutions to the transmission problem \eqref{inho} near $O$. The space $\Lambda^{l,\delta}_{\eta}(\Sigma)$ is endowed with the norm
\begin{equation*}
\|g\|_{\Lambda^{l,\delta}_{\eta}(\Sigma)}:=\sup_{x\in\Sigma}\bigg\{\sum^{l}_{j=0} |x|^{\eta-\delta-l+j}|\nabla^{j}g(x)|\bigg\}+\sup_{x,y\in\Sigma}\bigg\{ \frac{||x|^{\eta}\nabla^{l}g(x)-|y|^{\eta}\nabla^{l}g(y)|}{|x-y|^{\delta}} \bigg\}.
\end{equation*}
Obviously, the weight $\eta\in\mathbb{N}$ characterizes the singularity at $O$. For more properties of the weighted H\"older spaces $\Lambda^{l,\delta}_{\eta}(\Sigma)$, we refer to \cite[Section 2]{Hu2020} and \cite{NP}.

Set $w_{0,D}=w-v_{0,D}\in C^{0,\delta}(\overline{\Sigma})\subset\Lambda_{0}^{0,\delta} (\Sigma)$, where $w$ fulfills the system \eqref{inho}. Then $w_{0,D}$ solves
\begin{equation} \label{w_D}
\left\{\begin{array}{lll}
  \Delta w_{0,D}=\widetilde{f}_{0},& \quad \mbox{in}\quad \Sigma, \\
  \big[w_{0,D}\big]=\big[\frac{\partial w_{0,D}}{\partial \nu}\big]=0, & \quad \mbox{on} \quad \Gamma_{0}, \\
  w_{0,D}=0, & \quad \mbox{on}\quad \Gamma^{+}\cup\Gamma^{-},
\end{array}\right.
\end{equation}
where $\widetilde{f}_{0}:=f^{\pm}-c_{0}^{\pm}$ in $\Sigma^{\pm}$.  Since $\widetilde{f}_{0}(O)=0$, we have $\widetilde{f}_{0}\in  \Lambda_{0}^{0,\delta}(\Sigma)\cap \Lambda_{1}^{0,\delta}(\Sigma)$ for some $\delta\in(0,1)$. Making use of an appropriate cut-off function, the above problem can be formulated in an infinite sector, in which the Dirichlet boundary value problem is uniquely solvable in a corresponding weighted H\"older space $\Lambda^{2, \delta}_1$; see \cite{NP}.  This gives the solution
$w_{0, D}\in \Lambda_{1}^{2,\delta}(\Sigma)$ with the asymptotics (see also \cite[Proposition 4]{Elschner2018})
\begin{equation*}
w_{0,D}=d_{D,2}r^{2}\sin(2\theta)+\mathcal{O}\big(r^{2+\delta}\big),\quad r\rightarrow0^{+}.
\end{equation*}
Note that here we have used the fact the opening angle of $\Sigma$ is $\pi$. Hence, as $r\rightarrow0^{+}$ in $\Sigma^{\pm}$,
\begin{equation*}
w=w_{0,D}+v_{0,D}=d_{D,2}r^{2}\sin(2\theta)+\mathcal{O}\big(r^{2+\delta}\big)+q^{\pm}_{2,D} +C_{0,D}r^{2}\big[\ln r\sin(2\theta)+\theta\cos(2\theta)\big].
\end{equation*}
Below we shall prove that a solution with the above asymptotic behavior near $O$ cannot fulfill the homogeneous Neumann boundary condition. In fact,  one can prove analogously that, as a solution to the Neumann boundary value problem, $w$ admits the asymptotics
\begin{equation*}
w=w_{0,N}+v_{0,N}=d_{N,2}r^{2}\cos(2\theta)+\mathcal{O}\big(r^{2+\delta}\big)+q^{\pm}_{2,N} +C_{0,N}r^{2}\big[\ln r\cos(2\theta)-\theta\sin(2\theta)\big].
\end{equation*}
Comparing the coefficients of the previous two identities, we find that
\begin{equation*}
C_{0,D}=C_{0,N}=0 \quad \mbox{and}\quad Q^{\pm}_{2,D}=Q^{\pm}_{2,N}:=Q^{\pm}_{2} \quad \mbox{in}~\Sigma,
\end{equation*}
where $Q_{2,D}^{\pm}:=d_{D,2}r^{2}\sin(2\theta)+q^{\pm}_{2,D}$, $Q^{\pm}_{2,N}:=d_{N,2}r^{2}\cos(2\theta)+q^{\pm}_{2,N}$. Furthermore, $Q^{\pm}_{2}$ satisfies the following problem (cf. \eqref{3.7}):
\begin{equation*}
\left\{\begin{array}{lll}
  \Delta Q^{\pm}_{2}=c_{0}^{\pm},& \quad \mbox{in}\quad \Sigma, \\
  Q^{+}_{2}=Q^{-}_{2},\quad \frac{\partial Q^{+}_{2}}{\partial \nu}=\frac{\partial Q^{-}_{2}}{\partial \nu}, & \quad \mbox{on}\quad \Gamma_{0}, \\
  Q^{\pm}_{2}=\frac{\partial Q^{\pm}_{2}}{\partial \nu}=0, & \quad \mbox{on}\quad \Gamma^{\pm}.
\end{array}\right.
\end{equation*}

By Lemma \ref{lem=}, we can see that $c_{0}^{+}=c_{0}^{-}$. In the following, we will prove that $c_{0}^{+}=c_{0}^{-}=0$. Since $u_{1}(O)=u_{2}(O):=u(O)$, we have
\begin{equation*}
\left\{\begin{array}{lll}
c_{0}^{+}=f^{+}(O)=k_{1}^{2}u_{2}(O)-k_{1,2}^{2}u_{1}(O)=(k_{1}^{2}-k_{1,2}^{2})u(O), \\
c_{0}^{-}=f^{-}(O)=k_{2,2}^{2}u_{2}(O)-k_{1}^{2}u_{1}(O)=(k_{2, 2}^{2}-k_{1}^{2})u(O).
\end{array}\right.
\end{equation*}
By our assumptions on $k_{1,2}$ and $k_{2,2}$, we conclude that $c_0^+$ and $c_0^-$ have different signs if $u(O)\neq 0$.
Combining with the identity $c_{0}^{+}=c_{0}^{-}$, we obtain that $c_{0}^{+}=c_{0}^{-}=0$ and then $u(O)=0$.

Recalling the representation of the functions $u_{1}$ and $u_2$ in (\ref{uj}), we achieve that $b^{(1)}_{0}=b^{(2)}_{0}=0$ and thus as $r\rightarrow 0$,
\ben
&f^{+}(r,\theta)=k_{1}^{2}u_{2}-k_{1,2}^{2}u_{1}= r(c_{1,a}^{+}\sin\theta+c_{1,b}^{+}\cos\theta)+\mathcal{O}(r^{2}),\\
&f^{-}(r,\theta)=k_{2,2}^{2}u_{2}-k_{1}^{2}u_{1}=r(c_{1,a}^{-}\sin\theta+c_{1,b}^{-}\cos\theta)+\mathcal{O}(r^{2}),
\enn
where
\be\label{cab}
\begin{split}
&&c_{1,a}^{+}:=k_{1}^{2}\,a^{(2)}_{1}-k_{1,2}^{2}\,a^{(1)}_{1}, \quad
c_{1,b}^{+}:=k_{1}^{2}\,b^{(2)}_{1}-k_{1,2}^{2}\,b^{(1)}_{1}, \\
&&c_{1,a}^{-}:=k_{2,2}^{2}\,a^{(2)}_{1}-k_{1}^{2}\,a^{(1)}_{1},\quad
c_{1,b}^{-}:=k_{2,2}^{2}\,b^{(2)}_{1}-k_{1}^{2}\,b^{(1)}_{1}.
\end{split}
\en

{\bf Step 2:} Prove that $c^\pm_{1,a}=c^\pm_{1,b}=0$ and $a_1^{(j)}=b_1^{(j)}=0$ for $j=1,2$. This step is not necessary for carrying out our induction arguments in the next Step 3. However, for the readers' convenience we still keep it here.

As done in Step 1, we consider the Dirichlet and Neumann problems separately by replacing the right hand side by its lowest order term.
Consider the problems
\begin{equation} \label{Neu}
\left\{\begin{array}{lll}
  \Delta v_{1,D}=r(c_{1,a}^{\pm}\sin\theta+c_{1,b}^{\pm}\cos\theta),& \quad \mbox{in}\quad\Sigma^{\pm}, \\
  \big[v_{1,D}\big]=\big[\frac{\partial v_{1,D}}{\partial \nu}\big]=0, & \quad \mbox{on} \quad\Gamma_{0}, \\
  v_{1,D}=0, & \quad \mbox{on}\quad\Gamma^{+}\cup\Gamma^{-},
\end{array}\right.
\end{equation}
\begin{equation} \label{Neu2}
\left\{\begin{array}{lll}
  \Delta v_{1,N}=r(c_{1,a}^{\pm}\sin\theta+c_{1,b}^{\pm}\cos\theta),& \quad \mbox{in}\quad\Sigma^{\pm}, \\
  \big[v_{1,N}\big]=\big[\frac{\partial v_{1,N}}{\partial \nu}\big]=0, & \quad \mbox{on}\quad\Gamma_{0},\vspace{0.05cm} \\
  \frac{\partial v_{1,N}}{\partial\nu}=0, & \quad \mbox{on}\quad\Gamma^{+}\cup\Gamma^{-}.
\end{array}\right.
\end{equation}
By Lemma \ref{A3}, there exist two special solutions to (\ref{Neu}) and (\ref{Neu2}) of the form
\ben
&v_{1,D}=q^{\pm}_{3,D}(r,\theta)+C_{1,D}r^{3}\big[\ln r\sin(3\theta) +\theta\cos(3\theta)\big],\\
&v_{1,N}=q^{\pm}_{3,N}(r,\theta)+C_{1,N}r^{3}\big[\ln r\cos(3\theta) -\theta\sin(3\theta)\big],
\enn
where $q^{\pm}_{3,D}(r,\theta)$ and $q^{\pm}_{3,N}(r,\theta)$ are homogeneous polynomials of degree three satisfying the systems (\ref{Neu}) and (\ref{Neu2}), respectively. Then $w_{1,D}:=w-v_{1,D}$ solves the problem (\ref{w_D}) with the right term $\widetilde{f}_{1}:=f^{\pm}-r(c_{1,a}^{\pm}\sin\theta+c_{1,b}^{\pm}\cos\theta)$ in $\Sigma^{\pm}$. Since $\widetilde{f}_{1}(O)=|\nabla \widetilde{f}_{1}(O)|=0$, we can see that $\widetilde{f}_{1}\in \Lambda_{-1}^{0,\delta}(\Sigma)\cap \Lambda_{0}^{0,\delta}(\Sigma)$, which implies that $w_{1,D}\in\Lambda_{0}^{2,\delta}(\Sigma)$. Hence, $w_{1,D}$ takes the form
\begin{equation*}
w_{1,D}=d_{D,3}r^{3}\sin(3\theta)+\mathcal{O}\big(r^{3+\delta}\big),\quad r\rightarrow 0^{+},
\end{equation*}
and then
\begin{equation*}
w=w_{1,D}+v_{1,D}=d_{D,3}r^{3}\sin(3\theta)+\mathcal{O}\big(r^{3+\delta}\big)+q^{\pm}_{3,D} +C_{1,D}r^{3}\big[\ln r\sin(3\theta)+\theta\cos(3\theta)\big].
\end{equation*}
Similarly,
\begin{equation*}
w=w_{1,N}+v_{1,N}=d_{N,3}r^{3}\cos(3\theta)+\mathcal{O}\big(r^{3+\delta}\big)+q^{\pm}_{3,N} +C_{1,N}r^{3}\big[\ln r\cos(3\theta)-\theta\sin(3\theta)\big].
\end{equation*}
Comparing the coefficients of the above two identities, we find
\begin{equation*}
C_{1,D}=C_{1,N}=0 \quad \mbox{and} \quad Q^{\pm}_{3,D}=Q^{\pm}_{3,N}=:Q^{\pm}_{3}
\end{equation*}
where $Q^{\pm}_{3,D}:=d_{D,3}r^{3}\sin(3\theta)+q^{\pm}_{3,D}$, $Q^{\pm}_{3,N}:=d_{N,3}r^{3}\cos(3\theta)+q^{\pm}_{3,N}$ and $Q^{\pm}_{3}$ satisfies:
\begin{equation*}
\left\{\begin{array}{lll}
  \Delta Q^{\pm}_{3}=r(c_{1,a}^{\pm}\sin\theta+c_{1,b}^{\pm}\cos\theta),& \quad \mbox{in}\quad \Sigma^\pm, \\
  Q^{+}_{3}=Q^{-}_{3},\quad \frac{\partial Q^{+}_{3}}{\partial \nu}=\frac{\partial Q^{-}_{3}}{\partial \nu}, & \quad \mbox{on}\quad \Gamma_{0}, \\
  Q^{\pm}_{3}=\frac{\partial Q^{\pm}_{3}}{\partial \nu}=0, & \quad \mbox{on}\quad \Gamma^{\pm}.
\end{array}\right.
\end{equation*}
Using again Lemma \ref{lem=}, we obtain that $Q_{3}^{+}=Q_{3}^{-}$. Hence, $c_{1,a}^{+}\sin\theta+c_{1,b}^{+}\cos\theta=c_{1,a}^{-}\sin\theta+c_{1,b}^{-}\cos\theta$ for all $\theta\in(0, 2\pi)$, implying that $c_{1,a}^{+}=c_{1,a}^{-}$ and $c_{1,b}^{+}=c_{1,b}^{-}$.

Next, we will prove $c_{1,a}^{\pm}=c_{1,b}^{\pm}=0$. In view of the transmission conditions at $\theta=-\pi/2$ for all $r\in[0, R)$, we may set $\partial_{r}u_{1}(O) =\partial_{r}u_{2}(O)=:\partial_{r}u(O)$, $\partial_{\theta}\partial_{r}u_{1}(O) =\partial_{\theta}\partial_{r}u_{2}(O) =:\partial_{\theta}\partial_{r}u(O)$. In view of the definition of $c^\pm_{1,b}$ and $c^\pm_{1,a}$ we obtain
\begin{equation*}
\left\{\begin{array}{lll}
c_{1,b}^{+}=\partial_{r}f^{+}(O)=k_{1}^{2}\partial_{r}u_{2}(O)-k_{1,2}^{2}\partial_{r} u_{1}(O) =(k_{1}^{2}-k_{1,2}^{2})\partial_{r}u(O), \\
c_{1,b}^{-}=\partial_{r}f^{-}(O)=k_{2,2}^{2}\partial_{r}u_{2}(O)-k_{1}^{2}\partial_{r} u_{1}(O)=(k_{2,2}^{2}-k_{1}^{2})\partial_{r}u(O),\\
c_{1,a}^{+}=\partial_{\theta}\partial_{r}f^{+}(O)=k_{1}^{2}\partial_{\theta}\partial_{r}u_{2}(O)-k_{1,2}^{2}\partial_{\theta}\partial_{r} u_{1}(O)=(k_{1}^{2}-k_{1,2}^{2})\partial_{\theta}\partial_{r}u(O), \\
c_{1,a}^{-}=\partial_{\theta}\partial_{r}f^{-}(O)=k_{2,2}^{2}\partial_{\theta}\partial_{r}u_{2}(O)-k_{1}^{2}\partial_{\theta}\partial_{r} u_{1}(O)=(k_{2,2}^{2}-k_{1}^{2})\partial_{\theta}\partial_{r}u(O).
\end{array}\right.
\end{equation*}
Recalling the assumptions of $k_{1,2}$ and $k_{2,2}$, we find that $k_1^2-k_{1,2}^2$ and $k_{2,2}^2-k_1^2$ have different signs.
Combining with the identity $c_{1,b}^{+}=c_{1,b}^{-}$, $c_{1,a}^{+}=c_{1,a}^{-}$, we obtain that
\begin{equation*}
c_{1,a}^{\pm}=c_{1,b}^{\pm}=0, \quad \partial_r u(O)=\partial_\theta\partial_r u(O)=0,
\end{equation*}
which together with \eqref{cab} yield $a^{(j)}_{1}=b^{(j)}_{1}=0$ for $j=1,2$.

{\bf Step 3:} Induction arguments. Making the induction hypothesis that
\begin{equation*}
a^{(1)}_{j}=a^{(2)}_{j}=b^{(1)}_{j}=b^{(2)}_{j}=0\quad\mbox{for all}\quad 0\leq j\leq n-1, \;n\geq2,
\end{equation*}
we will prove that $a^{(1)}_{n}=a^{(2)}_{n}=b^{(1)}_{n}=b^{(2)}_{n}=0$.

The induction hypothesis implies that as $r\rightarrow 0$,
\begin{align*}
&f^{+}(r, \theta)=k_{1}^{2}u_{2}-k_{1,2}^{2}u_{1}= r^n\,\big[c_{n,a}^{+}\sin(n\theta)+c_{n,b}^{+}\cos(n\theta)\big]
+\mathcal{O}(r^{2+n}),\quad\mbox{in}\quad \Sigma^+, \\
&f^{-}(r, \theta)=k_{2,2}^{2}u_{2}-k_{1}^{2}u_{1}=r^n\,\big[c_{n,a}^{-}\sin(n\theta)+c_{n,b}^{-}\cos(n\theta)\big]
+\mathcal{O}(r^{2+n}),\quad\mbox{in}\quad \Sigma^-,
\end{align*}
where
\ben
&& c_{n,a}^{+}:=k_{1}^{2}a^{(2)}_{n}-k_{1,2}^{2}a^{(1)}_{n},
\quad c_{n,b}^{+}:=k_{1}^{2}b^{(2)}_{n}-k_{1,2}^{2}b^{(1)}_{n}, \\
&& c_{n,a}^{-}:=k_{2,2}^{2}a^{(2)}_{n}-k_{1}^{2}a^{(1)}_{n},\quad
 c_{n,b}^{-}:=k_{2,2}^{2}b^{(2)}_{n}-k_{1}^{2}b^{(1)}_{n}.
\enn
Consider the problems
\begin{equation} \label{eq}
\left\{\begin{array}{lll}
  \Delta v_{n,D}=r^{n}\big[c_{n,a}^{\pm}\sin(n\theta)+c_{n,b}^{\pm}\cos(n\theta)\big],& \quad \mbox{in}\quad \Sigma^{\pm}, \\
  \big[v_{n,D}\big]=\big[\frac{\partial v_{n,D}}{\partial \nu}\big]=0, & \quad \mbox{on}\quad \Gamma_{0}, \\
  v_{n,D}=0, & \quad \mbox{on}\quad \Gamma^{+}\cup\Gamma^{-},
\end{array}\right.
\end{equation}
\begin{equation} \label{eq2}
\left\{\begin{array}{lll}
  \Delta v_{n,N}=r^{n}\big[c_{n,a}^{\pm}\sin(n\theta)+c_{n,b}^{\pm}\cos(n\theta)\big],& \quad \mbox{in}\quad \Sigma^{\pm}, \\
  \big[v_{n,N}\big]=\big[\frac{\partial v_{n,N}}{\partial \nu}\big]=0, & \quad \mbox{on} \quad \Gamma_{0}, \\
  \frac{\partial v_{n,N}}{\partial\nu}=0, & \quad \mbox{on} \quad \Gamma^{+}\cup\Gamma^{-}.
\end{array}\right.
\end{equation}
Recalling Lemma \ref{A3}, there exist two special solutions to problems (\ref{eq}) and (\ref{eq2}) of the form
\ben
&&v_{n,D}(r,\theta)=q_{n+2,D}^{\pm}(r,\theta)+C_{n,D}r^{n+2}\big\{\ln r\sin[(n+2)\theta] +\theta\cos[(n+2)\theta]\big\} \quad \mbox{in}~\Sigma^{\pm},\\
&&v_{n,N}(r,\theta)=q_{n+2,N}^{\pm}(r,\theta)+C_{n,N}r^{n+2}\big\{\ln r\cos[(n+2)\theta] -\theta\sin[(n+2)\theta]\big\} \quad \mbox{in}~\Sigma^{\pm},
\enn
where $q^{\pm}_{n+2,D}$ and $q^{\pm}_{n+2,N}$ are homogeneous polynomials of degree $n+2$ satisfying the system (\ref{eq}) and (\ref{eq2}), respectively. The function $w_{n,D}:=w-v_{n,D}$ then solves the problem (\ref{w_D}) with the right term
\begin{equation*}
\widetilde{f}_{n}:=f^{\pm}-r^{n}\big[c_{n,a}^{\pm}\sin(n\theta)+c_{n,b}^{\pm}\cos(n\theta)\big],\quad \mbox{in}\quad \Sigma^{\pm}.
\end{equation*}
Since $\partial_{r}^{l}\widetilde{f}_{n}(O)=0$ for all $0\leq l\leq n$, it holds that $\widetilde{f}_{n}\in \Lambda_{-n}^{0,\delta}(\Sigma)\cap \Lambda_{-n+1}^{0,\delta}(\Sigma)$, which implies that $w_{n,D},w_{n,N}\in \Lambda_{-n+1}^{2,\delta}(\Sigma)$ take the forms
\ben
&&w_{n,D}=d_{D,n+2}r^{n+2}\sin[(n+2)\theta]+\mathcal{O}\big(r^{n+2+\delta}\big),\\
&&w_{n,N}=d_{N,n+2}r^{n+2}\cos[(n+2)\theta]+\mathcal{O}\big(r^{n+2+\delta}\big),
\enn
as $r\rightarrow0$. Consequently,
\begin{align*}
w=&d_{D,n+2}r^{n+2}\sin[(n+2)\theta]+\mathcal{O}\big(r^{n+2+\delta}\big)+q^{\pm}_{n+2,D}\\
&+C_{n,D}r^{n+2}\big\{\ln r\sin[(n+2)\theta]+\theta\cos[(n+2)\theta]\big\} \\
=&d_{N,n+2}r^{n+2}\cos[(n+2)\theta]+\mathcal{O}\big(r^{n+2+\delta}\big)+q^{\pm}_{n+2,N}\\ &+C_{n,N}r^{n+2}\big\{\ln r\cos[(n+2)\theta]-\theta\sin[(n+2)\theta]\big\}.
\end{align*}
This implies the relations
\begin{equation*}
C_{n,D}=C_{n,N}=0 \quad \mbox{and} \quad Q^{\pm}_{n+2,D}=Q^{\pm}_{n+2,N}=:Q^{\pm}_{n+2}
\end{equation*}
where $Q^{\pm}_{n+2,D}:=d_{D,n+2}r^{n+2}\sin[(n+2)\theta]+q^{\pm}_{n+2,D}$, $Q^{\pm}_{n+2,N}:=d_{N,n+2}r^{n+2}\cos[(n+2)\theta]+q^{\pm}_{n+2,N}$ and $Q^{\pm}_{n+2}$ satisfies
\begin{equation*}
\left\{\begin{array}{lll}
  \Delta Q^{\pm}_{n+2}=r^{n}\big[c_{n,a}^{\pm}\sin(n\theta)+c_{n,b}^{\pm}\cos(n\theta)\big],& \quad \mbox{in}\quad \Sigma,\vspace{0.05cm} \\
  Q^{+}_{n+2}=Q^{-}_{n+2},\quad \frac{\partial Q^{+}_{n+2}}{\partial \nu}=\frac{\partial Q^{-}_{n+2}}{\partial \nu}, & \quad \mbox{on}\quad \Gamma_{0}, \vspace{0.1cm} \\
  Q^{\pm}_{n+2}=\frac{\partial Q^{\pm}_{n+2}}{\partial \nu}=0, & \quad \mbox{on}\quad \Gamma^{\pm}.
\end{array}\right.
\end{equation*}
By Lemma \ref{lem=}, we conclude that $Q_{n+2}^{+}=Q_{n+2}^{-}$ and then $c_{n,a}^{+}=c_{n,a}^{-}$, $c_{n,b}^{+}=c_{n,b}^{-}$.

Since $\partial_{r}^{n}u_{1}(O)=\partial_{r}^{n}u_{2}(O):=\partial_{r}^{n}u(O)$ and $\partial_{\theta}\partial_{r}^{n}u_{1}(O)=\partial_{\theta}\partial_{r}^{n}u_{2}(O):=\partial_{\theta}\partial_{r}^{n}u(O)$, we have
\begin{equation*}
\left\{\begin{array}{lll}
c_{n,b}^{+}n!=\partial_{r}^{n}f_{n}^{+}(O)=k_{1}^{2}\partial_{r}^{n}u_{2}(O) -k_{1,2}^{2}\partial_{r}^{n}u_{1}(O)=(k_{1}^{2}-k_{1,2}^{2})\partial_{r}^{n}u(O), \\
c_{n,b}^{-}n!=\partial_{r}^{n}f_{n}^{-}(O)=k_{2,2}^{2}\partial_{r}^{n}u_{2}(O) -k_{1}^{2}\partial_{r}^{n}u_{1}(O)=(k_{2,2}^{2}-k_{1}^{2})\partial_{r}^{n}u(O),\\
c_{n,a}^{+}n!=\partial_{\theta}\partial_{r}^{n}f_{n}^{+}(O)=k_{1}^{2}\partial_{\theta}\partial_{r}^{n}u_{2}(O) -k_{1,2}^{2}\partial_{\theta}\partial_{r}^{n}u_{1}(O)=(k_{1}^{2}-k_{1,2}^{2})\partial_{\theta}\partial_{r}^{n}u(O), \\
c_{n,a}^{-}n!=\partial_{\theta}\partial_{r}^{n}f_{n}^{-}(O)=k_{2,2}^{2}\partial_{\theta}\partial_{r}^{n}u_{2}(O) -k_{1}^{2}\partial_{\theta}\partial_{r}^{n}u_{1}(O)=(k_{2,2}^{2}-k_{1}^{2})\partial_{\theta}\partial_{r}^{n}u(O).
\end{array}\right.
\end{equation*}
Again by the assumption of $k_{1,2}$ and $k_{2,2}$, we get
\begin{equation*}
c_{n,a}^{\pm}=c_{n,b}^{\pm}=0,\qquad \partial_r^n u(O)=\partial_\theta\partial_r^n u(O)=0,
\end{equation*}
which imply $a^{(j)}_{n}=b^{(j)}_{n}=0$ for $j=1,2$.

{\bf Step 4:} The final contradiction. The induction argument in the last step gives $a^{(j)}_{n}=b^{(j)}_{n}=0$ for $j=1,2$ and all $n\geq 0$. Using the second assertion of Corollary \ref{Cor1}, we deduce that $u_{1}=u_{2}\equiv0$ in $\Sigma$ and thus by unique continuation $u_{1}=u_{2}\equiv0$ in $\mathbb{R}^{2}$. Again using the arguments at the end of Case one, one can get a contradiction. This proves the coincidence of the grating files $\Lambda_{1}=\Lambda_{2}$ in Case two.

\vspace{5mm}

\subsection{Case three}
Assume there exists a corner $O$ of $\Lambda_2$ such that $O\in\Lambda_1$, but $O$ is not a corner point of $\Lambda_1$. Without loss of generality, we suppose that $O$ is located on a vertical line segment of $\Lambda_1$; see Figure \ref{fig3}.
\begin{figure}[h] \label{fig3}
 \begin{center}
  \includegraphics[width=5cm,height=3.5cm]{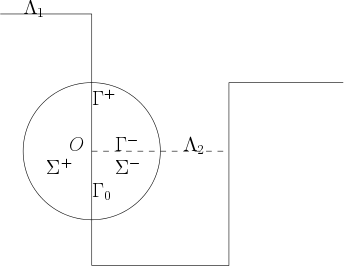}
  \end{center}
  \caption{Case three: $O\in\Lambda_1\cap\Lambda_2$ is a corner of $\Lambda_2$ but not a corner of $\Lambda_1$.}
\end{figure}

Choose $R>0$ sufficiently small such that the disk $B_{R}:=\{x\in \mathbb{R}^{2}:|x|<R\}$ does not contain other corners. Set
\begin{equation*}
B_{R}\cap \Lambda_{1}=\Gamma^{+}\cup\Gamma_{0},\quad B_{R}\cap \Lambda_{2}=\Gamma^{+}\cup\Gamma^{-},\quad \Sigma^{+}=B_{R}\cap \Omega_{\Lambda_{1}}^{-},\quad \Sigma^{-}=B_{R}\cap \Omega_{\Lambda_{2}}^{-}\cap \Omega_{\Lambda_{1}}^{+}.
\end{equation*}
We can see that $u_{1},u_{2}\in H^{2}(B_{R})\cap C^{0,\delta}(B_{R})$ ($0<\delta<1$) are solutions to the system
\begin{equation*}
\left\{\begin{array}{lll}
\Delta u_{1}+k_{1,2}^{2}u_{1}=0,\quad \Delta u_{2}+k_{2,2}^{2}u_{2}=0,\quad &\mbox{in}\quad \Sigma^{+}, \vspace{0.2cm} \\
\Delta u_{1}+k_{1}^{2}u_{1}=0,\quad  \Delta u_{2}+k_{2,2}^{2}u_{2}=0,\quad &\mbox{in}\quad \Sigma^{-}, \vspace{0.2cm}\\
\big[u_{1}\big]=\big[\frac{\partial u_{1}}{\partial\nu}\big]=0,\quad \big[u_{2}\big]=\big[\frac{\partial u_{2}}{\partial\nu}\big]=0, & \mbox{on}\quad \Gamma_{0}, \vspace{0.2cm}\\
u_{1}=u_{2}, \quad \frac{\partial u_{1}}{\partial\nu}=\frac{\partial u_{2}}{\partial\nu}, \quad & \mbox{on}\quad \Gamma^{+}\cup\Gamma^{-}.
\end{array}\right.
\end{equation*}
In contrast to Case two, the opening angle formed by $\Sigma^+\cup\Sigma^-\cup \Gamma_0$ is $3\pi /2$ rather than $\pi$. However, the arguments for treating Case two can be adapted to Case three. With slight modifications we can also deduce a contradiction. We omit the details for brevity. The proof of $\Lambda_1=\Lambda_2$ is thus complete.
\begin{remark}
If the near-field data are measured on two line segments above and below the grating, then we don't need to consider the Case three.
\end{remark}

\section{Proof of Theorem \ref{Main}: determination of refractive indices}
\label{medium}

Having uniquely determined the grating profies $\Lambda_1=\Lambda_2:=\Lambda$, we shall prove in this section that $k_{1,2}=k_{2,2}$. From $u_1(x_1, b)=u_2(x_1,b)$ for $x_1\in(0,2\pi)$, we get $u_1=u_2$ in $\Omega^+_\Lambda$. Choose a corner point $O\in\Lambda$ and $R>0$ sufficiently small, and set $\Pi=B_R\cap \Lambda$, $\Sigma^\pm=B_R\cap \Omega_\Lambda^\pm$. It is easy to see
\ben
&&\Delta u_1+k_{1,2}^2 u_1=0,\qquad \Delta u_2+k_{2,2}^2 u_2=0,\qquad\mbox{in}\quad \Sigma^-,\\
&&u_1=u_2,\quad \partial_\nu u_1=\partial_\nu u_2,\quad\mbox{on}\quad \Pi.
\enn
Note that the opening angle of $\Sigma^-$ is $\pi/2$ or $3\pi/2$.
Setting $w=u_1-u_2\in H^2(B_R)$, we get
\ben
&&\Delta w=f\quad\mbox{in}\quad \Sigma^-,\qquad f:=-k_{1,2}^2 u_1+k_{2,2}^2 u_2,\\
&&w=\partial_\nu w=0\quad\mbox{on}\quad \Pi.
\enn
Using the second assertion of Corollary \ref{Cor1}, we may assume that
\be\label{uj1}
u_j=\sum_{n\geq m}r^{n}\big[a_{n}^{(j)}\sin(n\theta)+b^{(j)}_{n}\cos(n\theta)\big]+\mathcal{O}(r^{m+2})\quad \mbox{as }r\rightarrow0^{+}, ~a^{(j)}_{n},b^{(j)}_{n}\in\mathbb{C},
\en
for some $m\geq0$ such that $|a^{(j)}_m|+|b^{(j)}_m|\neq 0$.
Otherwise, it holds that $u_1=u_2\equiv 0$ and a contradiction can be derived following the arguments at the end of Subsection \ref{case-one}.
We remark that, since $u_1=u_2$ in $\Sigma^+$, it holds in \eqref{uj1} that $a_m^{(1)}=a_m^{(2)}:=a_m$, $b_m^{(1)}=b_m^{(2)}:=b_m$ and that the index $m$ is uniform for $u_1$ and $u_2$. Hence, the right hand side admits the asymptotics
\ben
f(r, \theta)=r^m\big[c_m^+\,\sin(m\theta)+c_m^-\,\cos(m\theta)\big]
+\mathcal{O}(r^{m+2}), \quad r\rightarrow0, \quad \theta\in(0, 2\pi]
\enn with
\ben
c_m^+=-(k_{1,2}^2-k_{2,2}^2)\, a_{m},\qquad
c_m^-=-(k_{1,2}^2-k_{2,2}^2)\, b_{m}.
\enn
Since the lowest order term in the Taylor expansion of $f$ around $O$ is harmonic, applying \cite[Lemma 2.3]{Hu2020} gives the relation $c_m^\pm=0$. Since $|a_m|+|b_m|\neq 0$, we obtain $k_{1,2}=k_{2,2}$. The proof is complete.

\section{Appendix: well-posedness of forward scattering problem}\label{appendix}
In this section we prove well-posedness of our forward scattering problem under a more general transmission condition, which include both TE and TM polarizations.  The uniqueness proof seems new and of independent interests, since it applies to all frequencies, including  Rayleigh frequencies (which are also known as Wood anomalies), that is, $\beta_n^\pm =0$ for some $n\in \mathbb{Z}$.

For notational convenience we set $k_+=k_1$, $k_-=k_2$, $k(x)=k_\pm$ in $\Omega_\Lambda^\pm$. Consider the scattering problem
\begin{equation} \label{model}
\left\{\begin{array}{lll}
 \Delta u+k_{\pm}^{2}u=0,& \quad \mbox{in} \quad \Omega_{\Lambda}^{\pm}, \vspace{0.1cm} \\
 u^+=u^-, \quad \frac{\partial u^+}{\partial\nu}=\lambda\, \frac{\partial u^-}{\partial\nu},  & \quad \mbox{on}\quad \Lambda,\vspace{0.1cm} \\
 u=u^{i}+u^{s},& \quad \mbox{in} \quad \Omega_{\Lambda}^{+},
 \end{array}\right.
\end{equation}
where $\lambda>0$ is a constant, the notation $[\cdot]^\pm$ denotes the limit obtained from $\Omega_\Lambda^\pm$ and $\nu$ is the normal direction at $\Lambda$ pointing into $\Omega_\Lambda^+$.  The scattered field $u^s$ and the transmitted field $u$ are required to fulfill the upward and downward Rayleigh expansions \eqref{rad1} and \eqref{rad2}, respectively. We suppose that $\Lambda\in \mathcal{A}$ is a rectangular grating that satisfies the condition \eqref{gc}.  If $\Lambda$ is given by the graph of some function or $\mbox{Im}\, k_2>0$ (that is, the medium below $\Lambda$ is lossy), uniqueness and existence of the above transmission problem have been investigated in details; see e.g., \cite{Tilo-hab, D93, ES98, S98} in periodic structures and \cite{Hu2008, Thomas06} for rough interfaces.

\begin{theorem} Let $H>\max\{|\Lambda^+|, |\Lambda^-|\}$ and
suppose that one of the following conditions holds:
\[ (i)\; \lambda\geq 1,~k_+^2>\lambda\, k_-^2;\qquad (ii)\; \lambda\leq 1,~k_+^2<\lambda\, k_-^2. \]
Then the scattering problem \eqref{model} has a unique solution $u\in H_\alpha^1(S_H)$.\end{theorem}
\begin{proof}
Introduce the notations
\ben
S_H^\pm=\{x\in \Omega_\Lambda^\pm: -H<x_2<H\},\qquad \Gamma_H^\pm=\{(x_1, \pm H): 0<x_1<2\pi\}.
\enn
Define the DtN mappings $T^\pm: H_\alpha^{1/2}(\Gamma_H^\pm)\rightarrow H_\alpha^{-1/2}(\Gamma_H^\pm)$ by
\ben
(T^\pm f)(x_1):=\pm\sum_{n\in \mathbb{Z}} i\,\beta_n^\pm f_n\,e^{i\alpha_n x_1},\qquad f(x_1)=\sum_{n\in \mathbb{Z}} f_n\,e^{i\alpha_n x_1}\in H_\alpha^{1/2}(\Gamma_H^\pm).
\enn One may deduce from the above definitions that
\be\label{re}
&&\mbox{Re}\,\langle \pm T^\pm\,f, f\rangle=-\sum_{ |\alpha_n|> k_\pm} |\beta_n^\pm|\, |f_n|^2\leq 0,\\ \label{im}
&&\mbox{Im}\,\langle \pm T^\pm\,f, f\rangle=\sum_{ |\alpha_n|\leq k_\pm} |\beta_n^\pm|\, |f_n|^2\geq 0,
\en
where the pair $\langle \cdot, \cdot\rangle$ denotes the duality between $H_\alpha^{-1/2}$ and $H_\alpha^{1/2}$ on $\Gamma_H^\pm$.
Define a piecewise constant function $a(x):=1$ in $S_H^+$ and $a(x):=\lambda$ in $S_H^-$. The variational formulation for the scattering problem can be written as: find $u\in H_\alpha^1(S_H)$ such that for all $v\in H_\alpha^1(S_H)$,
\be\nonumber
&&\int_{S_H} \left[a(x) \nabla u\cdot \nabla \overline{v}- a(x) k(x) u\overline{v}\right]\,\mbox{d}x-\int_{\Gamma_H^+} T^+u \overline{v}\,\mbox{d}s+\lambda\,\int_{\Gamma_H^-} T^-u \overline{v}\,\mbox{d}s\\ \label{va}
&=&\int_{\Gamma_H^+}\left(T^+ u^i-\frac{\partial u^i}{\partial x_2}\right)\overline{v}\,\mbox{d}s.
\en
Using $\eqref{re}$, one can easily prove that the above sesquilinear form is strongly elliptic (see e.g., \cite{Tilo-hab, D93, ES98, S98}), giving rise to a Fredholm operator with index zero over $H_\alpha^{1/2}(S_H)$. By Fredholm alternative, it suffices to prove uniqueness. Suppose that $u^i\equiv 0$. Then $u$ satisfies the upward and downward Rayleigh expansion radiation conditions. Taking the imaginary part on both sides of \eqref{va} with $v=u$ and using \eqref{im}, we get
\ben
0=-\sum_{ |\alpha_n|\leq k_+} |\beta_n^+|\, |A^+_n|^2
-\lambda\,\sum_{ |\alpha_n|\leq k_-} |\beta_n^-|\, |A^-_n|^2,
\enn which implies the vanishing of the Rayleigh coefficients $A_n^\pm=0$ for $|\alpha_n|< k_\pm$. Taking the real part on both sides of \eqref{va} with $v=u$ and $u^i=0$ and using \eqref{re}, we obtain
\ben
I_1&:=&\int_{S_H} \left[a(x) |\nabla u|^2- a(x) k^2(x) |u|^2\right]\,\mbox{d}x\\
&=&\mbox{Re}\left\{\int_{\Gamma_H^+} T^+u \overline{u}\,\mbox{d}s- \lambda\,\int_{\Gamma_H^-} T^-u \overline{u}\,\mbox{d}s\right\} \\
&=& -\sum_{|\alpha_n|> k_+} |\beta_n^+|\, |A_n^+|^2\, e^{-2|\beta_n^+|\,H}-\lambda \sum_{|\alpha_n|> k_-} |\beta_n^-|\, |A_n^-|^2\, e^{-2|\beta_n^-|\,H}\\
&\leq& 0.
\enn
Multiplying the Helmholtz equation by $(x_2-c)\partial_2\overline{u}$ and integrating by part yield the Rellich's identities (\cite{Tilo-hab, CM, Hu2008, Thomas06}):
\be\nonumber
0&=&\left(\int_{\Gamma_H^\pm}\mp\int_\Lambda \right) (x_2-c)\left[
-\nu_2 |\nabla u^\pm |^2+\nu_2 k_\pm^2 |u|^2+2\mbox{Re}(\partial_2 \overline{u}^\pm\,\partial_\nu u^\pm)
\right]\,\mbox{d}s \\ \nonumber
&&+\int_{S_H^\pm} |\nabla u|^2-k_\pm^2\,|u|^2-2|\partial_2 u|^2\, \mbox{d}x\\ \nonumber
&:=&I^\pm,
\en
where the normal directions at $\Gamma_H^\pm$ are supposed to point into the exterior of $S_H$.
We remark that the integrals on the vertical boundaries of $\partial S_H$ have been canceled due the quasi-periodicity of $u$.
 The integrand over $\Lambda$ is well-defined because,  for rectangular gratings it holds that $u\in H^{3/2+\epsilon}_\alpha(S_H)$ for some $\epsilon>0$ depending on $\lambda$ (see e.g., \cite[Chapter 2.4.3]{Pe2001} and \cite[Section 3.3]{ES98} ).
Straightforward calculations show that
\ben
&&\int_{\Gamma_H^\pm} (x_2-c)\left[
-\nu_2 |\nabla u^\pm |^2+\nu_2 k_\pm^2 |u|^2+2\mbox{Re}(\partial_2 \overline{u}^\pm\,\partial_\nu u^\pm)
\right]\,\mbox{d}s \\
&=&(\pm H-c)\,\sum_{ |\alpha_n|\leq k_\pm} |\beta_n^\pm|\, |A^\pm_n|^2
= 0,
\enn
and (see e.g., \cite[Section 4]{Hu2008} and \cite[Chapter 2.4]{Tilo-hab} for details)
\be\nonumber
0&=&I^++\lambda\, I^-\\\nonumber
&=&-\int_\Lambda \left[ \lambda (\lambda-1)|\partial_\nu u^-|^2+(\lambda-1)|\partial_\tau u^-|^2+(k_+^2-\lambda k_-^2) |u|^2 \right] \nu_2 (x_2-c)\,\mbox{d}s\\ \label{identity}
&&-2\int_{S_H}a(x) |\partial_2 u|^2\,\mbox{d}x+ I_1,
\en where $\partial_\tau$ denotes the tangential derivative on $\Lambda$ with $\tau:=(-\nu_2, \nu_1)$.
By the assumptions on $k_\pm$, $\lambda$ and recalling the fact that $\nu_2\geq 0$ on $\Lambda$, we can always choose $c\in \mathbb{R}$ to ensure that the integral over $\Lambda$ is non-positive, so that each term in the above expression vanishes. Consequently, we get $\partial_2 u\equiv 0$ in $S_H$ and $I_1=0$, implying that $A_n^\pm=0$ for all $|\alpha_n|>k_\pm$. Therefore,
\ben
u=A_n^\pm e^{ik_\pm x_1}+ B_m^\pm\,e^{-ik_\pm x_1}\quad\mbox{in}\quad \Omega_\Lambda^\pm,\qquad A_n^\pm, B_m^\pm\in \mathbb{C},
\enn if $\alpha_n=k_\pm$ or $\alpha_m=-k_\pm$ for some $n,m\in \mathbb{Z}$ (that is, Rayleigh frequencies occurs). Note that the above expression of $u$ is well-defined in $\mathbb{R}^2$.
Since $\nu_2=1$ on the line segment of $\Lambda$ parallel to the $x_1$-axis and $|k_+^2-\lambda k_-^2|>0$, one can also deduce from \eqref{identity} that $u \equiv 0$ on this segment, which gives $A_n^\pm=B_m^\pm=0$ and thus $u\equiv 0$.
\end{proof}
In the special case that $\lambda=1$ (i.e., TE polarization), we get well-posedness of our scattering problem \eqref{a}-\eqref{rad2}.

\begin{corollary}
Let $\Lambda\in \mathcal{A}$ be a rectangular penetrable grating and  assume $k_2\in \mathbb{R}$, $k_2\neq k_1$.
The direct scattering problem \eqref{a}-\eqref{rad2} has a unique solution  $u\in H_\alpha^2(S_H)$ for any fixed $H>\max\{|\Lambda^+|, |\Lambda^-|\}$.
\end{corollary}

\section{Concluding remarks}

In this paper, we have verified the uniqueness in identifying a penetrable rectangular grating profile and the material parameter from a single measurement taken above the grating. We remark that, since only local regularity properties of the Helmholtz equation are involved, the uniqueness results carry over to any incoming wave, provided the forward problem is solvable in appropriate Sobolev spaces. Further, the uniqueness remain valid if $k_2\in\mathbb{C}$ and ${\rm Im}\,k_2\geq0$. On the other hand, we observe that the $2\pi$-periodicity assumption on the scattering surface can be removed. For non-periodic rectangular interfaces satisfying \eqref{gc}, well-posedness of the forward scattering can be established following the variational arguments in \cite{CM, Hu2008,Thomas06} for treating rough surfaces.  In addition, our arguments provide insights into the corner scattering theory in a non-convex domain.  The TE transmission conditions lead to  a good solution regularity that $u\in H^2(S_H)$, which however cannot hold true in the TM polarization case.
 In the future, we will discuss the inverse problem under the more general transmission boundary condition such as $\partial u_{+}/\partial\nu=\lambda\partial u_{-}/\partial\nu$ ($\lambda\neq1$) (which covers the TE polarization case when $\lambda=(k_-/k_+)^2$) and also consider the complex-valued refractive index function. Further efforts will be made to extend the uniqueness results to these scattering problems.

\vspace{0.3cm}
{\bf Acknowledgments.}

The work of G. Hu is supported by the National Natural Science Foundation of China (No. 12071236) and the Fundamental Research Funds for Central Universities in China (No. 63213025). The authors would like to thank J. Elschner for helpful discussions on the inhomogeneous Laplacian equation \eqref{inho}.

\end{document}